\documentclass[11pt]{article}

\usepackage[english]{babel}
\usepackage{bm}
\usepackage[letterpaper,top=2cm,bottom=2cm,left=3cm,right=3cm,marginparwidth=1.75cm]{geometry}

\usepackage{amsmath}
\usepackage{amssymb}
\usepackage{amsthm}
\usepackage{graphicx}
\usepackage[colorlinks=true, allcolors=blue]{hyperref}

\newtheorem{proposition}{Proposition}[section]

\newtheorem{theorem}{Theorem}[section]
\newtheorem{lemma}[proposition]{Lemma}
\newtheorem{definition}[proposition]{Definition}
\newtheorem*{remark}{Remark}

\title{Construction of periodic solutions of multi-dimensional nonlinear wave equations with unbounded perturbation}
\author{Huining Xue \footnotemark[1] $^1$,  Zuhong You \thanks{The work was supported  by National Natural Science Foundation of China ( Grant NO. 12371189).} $^1$ and Xiaoping Yuan \footnotemark[1] $^1$  \\
		$^{1}$School of Mathematical Sciences, Fudan University, Shanghai 200433, P. R. China   }
\date{}
\begin{document}
\maketitle
\begin{abstract}
    By applying the Craig-Wayne-Bourgain (CWB) method, we establish the persistence of periodic solutions to multi-dimensional nonlinear wave equations (NLW) with unbounded perturbation.
\end{abstract}
\tableofcontents

\section{Introduction}
Over the past few decades, significant progress has been made in the KAM (Kolmogorov-Arnold-Moser) theory for nonlinear Hamiltonian partial differential equations (PDEs). The first existence results were established by Kuksin\cite{Kuk87} and Wayne\cite{Way90}, who studied nonlinear wave (NLW) and nonlinear  Schr\"{o}dinger equations (NLS) in one spatial dimension under Dirichlet boundary conditions.
Two main approaches have emerged: the classical KAM technique (e.g., \cite{Kuk87}, \cite{Way90},  \cite{Kuk93}, \cite{Pos96}, \cite{LY00}, \cite{eliasson2010kam},  \cite{yuan2021kam}) and the Craig-Wayne-Bourgain (CWB) method (e.g.,  \cite{craig1993newton}, \cite{bourgain1994construction},\cite{bourgain1995construction}, \cite{bourgain1998quasi},\cite{bourgain2005green}, \cite{BP11duke}, \cite{berti2013quasi}, \cite{wang2016energy}, \cite{BB20} ). There are too many works to list here in this field. 

When considering unbounded perturbations, where the nonlinearity involves derivatives,  KAM theory faces significant challenges. The first results for unbounded perturbations were pioneered by Kuksin \cite{Kuk98kdv},\cite{Kuk00},  who studied small denominator equations with large variable coefficients and developed applicable KAM theorems to investigate the persistence of finite-gap solutions for the KdV equation. The estimate concerning the small-denominator equation with
large variable coefficients is now called the Kuksin lemma\cite{Kuk87}. 
Later, Liu and Yuan\cite{LY10spec}\cite{LY11dnls} extended the Kuksin lemma to the limiting case and established the KAM theorem for quantum Duffing oscillators, derivative nonlinear Schrödinger equations (DNLS) and and Benjamin-Ono equations. The Italian school (e.g., \cite{BBP13dnlw}, \cite{BBM14}, \cite{BBM16kdv}, \cite{BBE18}) has developed a novel approach for addressing unbounded perturbations.
It is worth mentioning that all these approaches are restricted to the case where the spatial dimension is equal to $1$ and rely on classical KAM techniques.

In our recent work on multi-dimensional nonlinear Schr\"{o}dinger equations (NLS),  we applied the CWB method to study periodic response solutions of NLS with fractional derivative perturbations (unbounded ones). In this paper, we extend this approach to multi-dimensional nonlinear wave equations (NLW) with fractional derivative perturbations (unbounded ones), where we improve the separation lemma for $|(n\lambda)^{2}-|m|^{2}|$ to handle unbounded perturbations, and explore the persistence of periodic solutions.

We consider the multi-dimensional Nonlinear Wave equations (NLW)
\begin{equation}
    y_{tt}-\Delta y+ \rho y+D^{\alpha} (y^{3})=0,
    \label{nlw}
\end{equation}
under the periodic boundary condition (i.e., $x\in\mathbb{T}^{d}$).
The fractional derivative $D^{\alpha}$ is defined as follows:

Let  $f\in L^{2}(\mathbb{T}^{d})$ be a smooth periodic function on the torus $\mathbb{T}^{d}:=(\mathbb{R}/2\pi \mathbb{Z})^{d}$.  We define
\begin{equation}
    D^{\alpha} f(x)=\sum\limits_{m\in \mathbb{Z}^{d}} \langle m\rangle^{\alpha}\hat{f}(m)e^{im\cdot x},
\end{equation}
where $\hat{f}(m)$ is the $m$-th Fourier coefficient of $f$,  
and  $\langle m\rangle=\sqrt{\sum\limits_{1\le j\le d} |m_{j}|^{2}+1}$. Note that we define the norm of $f\in L^{2}(\mathbb{T}^{d})$ as  $\lVert f\rVert=\frac{1}{(2\pi)^{d}}\int_{\mathbb{T}^{d}} |f(x)|^{2}dx$. Thus, $\{e^{im\cdot x}\}_{m\in \mathbb{Z}^{d}}$ is a standard orthonormal basis of $L^{2}(\mathbb{T}^{d})$, and 
\[\hat{f}(m)=\frac{1}{(2\pi)^{d}}\int_{\mathbb{T}^{d}} f(x)e^{-i m\cdot x}dx.\]

Replacing $y$ by $\varepsilon y$, equation (\ref{nlw}) becomes 
\begin{equation}
     y_{tt}-\Delta y+ \rho y+\varepsilon^{2} D^{\alpha} (y^{3})=0,
     \label{nlwper}
\end{equation}
which appears as a perturbation of the linear wave equation
\begin{equation}
    y_{tt}-\Delta y+ \rho y=0.
    \label{lw}
\end{equation}
It is straightforward that
\begin{equation}
    y=p \cos ( m\cdot x +\lambda t) \textup{ or } y=p \sin ( m\cdot x +\lambda t)
\end{equation}
with $\lambda=(|m|^{2}+\rho)^{1/2}$ are 
special solutions of (\ref{lw}). Fix $m_{0}\in\mathbb{Z}^{d}\setminus\{0\}$. Our aim is to show the persistence of the solution 
\begin{align}
    y_{0}=p_{0} \cos( m_{0}\cdot x +\lambda_{0}t),\\
    \lambda_{0}=(|m_{0}|^{2}+\rho)^{1/2},
\end{align}
of (\ref{lw}) for the perturbed equation (\ref{nlwper}). More precisely, we aim to find a  solution of (\ref{nlwper})  expressed in the form
\begin{equation}
    y(x,t)=\sum\limits_{m\in \mathbb{Z}^{d}, n\in\mathbb{Z}}\hat{y}(m,n)\cos ( m\cdot x   +n\lambda t ),
    \label{soltofind}
\end{equation}
where
\begin{align}
    &\hat{y}(m_{0}, 1)=\hat{y}(-m_{0},-1)=\frac{1}{2}p_{0},\\
    &\hat{y}(m,n)=\hat{y}(-m,-n),\\
    &\sum\limits_{\substack{(m,n)\ne(m_{0},1),\\(m,n)\ne (-m_{0},-1)}}|\hat{y}(m,n)|e^{|(m,n)|_{1}^{c}}<C\varepsilon^{\frac{1}{8}},\\
    &|\lambda-\lambda_{0}|<C\varepsilon.\label{frequency}
\end{align}
Here, $|(m,n)|_{1}=\sum\limits_{k=1}^{d}|m_{k}|+|n|$, $C$ only depends on $d$, $\alpha$, $m_{0}$. The constant $c$ is  sufficiently small.
We regard $p_{0}\in [1,2]$ as the parameter. 
We have the following theorem:
\begin{theorem}
Fix $m_{0}$ and suppose $\rho$ satisfies
\begin{align}
    &|n\rho-k|>\gamma |n|^{-2d}, \textup{ for } n,k\in\mathbb{Z},  n\ne 0, \\
     &\left| \sum\limits_{k=0}^{10d} a_{k}(\sqrt{\rho+|m_{0}|^{2}})^{k} \right|>\gamma\left(\sum |a_{k}|\right)^{-\tilde{C}}, \textup{ for all } \{a_{k}\}\in\mathbb{Z}^{10d+1}\setminus\{0\},\label{NONRE2}
\end{align}
where $\gamma$ is a small constant and $\tilde{C}$ is a large number depending on $d$.

Let $\alpha$  and $c$ be sufficient small. 
If $\varepsilon$ is sufficiently small, then there is a Cantor type set $I_{\varepsilon}\in[1,2]$. For $p_{0}\in I_{\varepsilon}$, the periodic solution 
\begin{equation}
    y_{0}=p_{0} \cos( m_{0}\cdot x +\lambda_{0}t)
\end{equation}
persists. The persisted periodic solution is of Gevrey smoothness and with frequency
\begin{equation}
    \lambda^{2}=|m_{0}|+\rho+\frac{3}{4}\langle m_{0} \rangle^{\alpha} p_{0}^{2}\varepsilon^{2}+O(\varepsilon^{\frac{17}{8}}).
\end{equation}
Moreover, we have
\begin{equation}
\textup{mes }([1, 2]\setminus I_{\varepsilon})\to 0, \textup{ as } \varepsilon\to 0.
\end{equation}
    \label{maintheorem}
\end{theorem}

\begin{remark}
    The order $\alpha$ in this paper is chosen to be very small and decreases rapidly as the spatial dimension $d$ increases. Therefore, we do not pursue the precise restriction of $\alpha$, as the focus of this paper is on the validity and robustness of the method rather than on fine-tuning the oeder $\alpha$.
\end{remark}

\textbf{Notations.} $\mathbb{Z}_{+}=\{1,2,...\}$.

$a\lesssim b$ refers to $a\le Cb$ for some constant depending on $m_{0}$, $d$, $\alpha$.

For $x\in\mathbb{R}^{b}$, denote $|x|=\sqrt{\sum\limits_{j=1}^{b}|x_{j}|^{2}}$ and $|x|_{1}=\sum\limits_{j=1}^{b}|x_{j}|$.

\textbf{Description of the paper.} The paper is organized as follows: In Section 2, we present the lattice formulation of the problem and introduce the Lyapounov-Schmidt decomposition. In Section 3, we focus on the analytic part of solving the $P$-equations without measure estimate. Section 4 is dedicated to solving the $Q$-equations and estimating the measure. We first solve the $Q$-equations and  provide  measure estimate in finite steps. Then, we give an improved separation lemma for $|(n\lambda)^{2}-|m|^{2}|$, followed by the measure estimate for infinite steps.  Finally, the appendices contain facts about some arithmetical conditions, the proof of the  improved separation lemma and the coupling lemmas.

\section{Lattice Formulation and Lyapounov-Schmidt decomposition}
\begin{definition}
    Let $y(x, t)$ be  $2\pi$ periodic in $x$ and $\frac{2\pi}{\lambda}$ periodic in $t$. We say $u(x, \theta)\in L^{2}(\mathbb{T}^{d+1})$ is the hull of $y(x,t)$, if we have
    \begin{equation}
        y(x,t)=u(x,\lambda t).
    \end{equation}
\end{definition}
We denote by $L^{2}_{e}$ the set of functions which are square integrable and even.
Suppose $u(x,\theta)\in L^{2}_{e}(\mathbb{T}^{d+1})$. Then we have
\begin{equation}
    \hat{u}(m,n)=\hat{u}(-m,-n),
    \label{ouhanshuxingshi}
\end{equation}
where $\hat{u}(m,n)=\frac{1}{(2\pi)^{2d+1}}\int_{\mathbb{T}^{d+1}}f(x,\theta)e^{-i(m\cdot x+n\theta)}$ refers to the Fourier coefficients.
To solve a periodic solution with frequency $\lambda$ of (\ref{nlwper})  is equivalent to solve
\begin{equation}
    (\lambda \partial_{\theta})^{2}u-\Delta u+\rho u+ \varepsilon^{2} D^{\alpha}(u^{3})=0.
    \label{hullformal}
\end{equation}
Denote the left side of the equation (\ref{hullformal}) by $F_{\lambda}(u)$. 
Note that if $u(x,\theta)$ is even, then $D^{\alpha}(u^{3})$ is also  even.
Passing to the Fourier coefficients, we have
\begin{equation}
    (-(n\lambda)^{2}+\mu_{m}^{2})\hat{u}(m,n)+\varepsilon^{2}\langle m \rangle^{\alpha}(u^{3})^{\land}(m,n)=0, \textup{ for } m\in \mathbb{Z}^{d},\  n\in\mathbb{Z},
    \label{quanbufangcheng}
\end{equation}
where 
\begin{equation}
    \mu_{m}=(|m|^{2}+\rho)^{1/2}.
\end{equation}
We denote the resonant set by $S$. Note that we want to find a solution with frequency $\lambda$ satisfying (\ref{frequency}).  Thus, we have
\begin{equation}
    S=\{(m,n)\in\mathbb{Z}^{d+1}: -(n\lambda_{0})^{2}+\mu_{m}^{2}=(-n^{2}|m_{0}|^{2}+|m|^{2})+(-n^{2}+1)\rho=0\}.
\end{equation}
Since $\rho$ satisfies a general Diophantine condition, we have
\begin{equation}
    S=\{(m,n):m\in\mathbb{Z}^{d},\ |m|=|m_{0}|,\  n=\pm 1\}.
\end{equation}
By restricting (\ref{quanbufangcheng}) to $S$ and $\mathbb{Z}^{d+1}\setminus S$, we obtain 
\begin{align}
    &Q\textup{-equations:} \    (-(n\lambda)^{2}+\mu_{m}^{2})\hat{u}(m,n)+\varepsilon^{2}\langle m \rangle^{\alpha}(u^{3})^{\land}(m,n)=0,\  (m,n)\in S,\label{Qeq}\\
    &P\textup{-equations:}\  (-(n\lambda)^{2}+\mu_{m}^{2})\hat{u}(m,n)+\varepsilon^{2}\langle m \rangle^{\alpha}(u^{3})^{\land}(m,n),\ \    (m,n)\in \mathbb{Z}^{d+1}\setminus S. \label{Peq}
\end{align}
The strategy is as follows. Fix $\lambda$, $\hat{u}(m_{0},1)=\hat{u}(-m_{0},-1)=\frac{1}{2}p_{0}$, $\hat{u}(m,1)=\hat{u}(-m,-1)=\frac{1}{2}p_{m}\ (|m|=|m_{0}|, \ m\ne m_{0})$. Using the $P$-equations, we solve $\hat{u}|_{\mathbb{Z}^{d+1}\setminus S}$ which depends on $$\lambda,\  p_{0},\  (p_{m})_{|m|=|m_{0}|, m\ne m_{0}}.$$
This will require restricting $(\lambda,  (p_{m})_{|m|=|m_{0}|})$ to a Cantor type set, constructed along the lines of the Newton scheme. The function $\hat{u}|_{\mathbb{Z}^{d+1}\setminus S}$ will be smoothly defined on the entire $(\lambda,  (p_{m})_{|m|=|m_{0}|})$ parameter set (see Section 3 for details of extension). Moreover, we would have $\hat{u}(m,n)=\hat{u}(-m,-n)$ for $(m,n)\in\mathbb{Z}^{d+1}$.  Next, substitute $\hat{u}|_{\mathbb{Z}^{d+1}\setminus S}$ in $Q$-equations
\begin{equation}
    (-\lambda^{2}+|m_{0}|^{2}+\rho)p_{m}+\varepsilon^{2}\langle m \rangle^{\alpha}(u^{3})^{\land}(m,1)=0,\  |m|=|m_{0}|,
\end{equation}
leading to equations in $(p_{m})_{|m|=|m_{0}|}$ and $\lambda$. Parameterizing in $p_{0}$, one obtains $p_{m}(p_{0}), \lambda(p_{0})$. Restrict  $p_{0}$ such that $(\lambda(p_{0}), p_{0}, (p_{m})_{|m|=|m_{0}|})$ is contained in the Cantor type set mentioned previously. For those $p_{0}$, we obtained the persisted solution.

\section{Solving the P-equations: the analytic part}
In this section, we construct approximate solutions $v_{j}$ for $P$-equations step by step.

Choose and fix constants $M ,c, C_{1}, C_{2}$ such that
\[
M>100 \langle m_{0}\rangle,\ \  0<c<\frac{\log \frac{17}{16}}{\log M},\ \  C_{1}>C_{2}>2.
\]
To describe the linearization of the $P$-equations, we introduce the notations as follows:
\begin{definition}
    Let $f\in L^{2}(\mathbb{T}^{d+1})$. We define the projective operators $\Gamma_{P}, \Gamma_{Q}$ as follows:
    \begin{align}
        &\Gamma_{P} f(x,\theta)=\sum\limits_{\substack{(m,n)\in\mathbb{Z}^{d+1}\\(m,n)\notin S}} \hat{f}(m,n)e^{i(m\cdot x+n\theta)},\\
        &\Gamma_{Q} f(x,\theta)=\sum\limits_{(m,n)\in S} \hat{f}(m,n)e^{i(m\cdot x+n\theta)}.
    \end{align}
\end{definition}
Moreover, let $N$ be an  integer. We define $\Gamma_{N}$ as follows:
\begin{equation}
    \Gamma_{N} f(x,\theta)=\sum\limits_{\substack{(m,n)\in\mathbb{Z}^{d+1}\\(m,n)\notin S\\|(m,n)|_{1}<N}} \hat{f}(m,n)e^{i(m\cdot x+n\theta)}.
\end{equation}
Let $u_{0}(x,\theta)=p_{0}\cos( m_{0}\cdot x  +\theta)+\sum\limits_{\substack{|m|=|m_{0}|\\ m\ne m_{0}}} p_{m} \cos( m\cdot x+\theta)$, where $p_{0}\in [1,2]$, $p_{m}\in [-1,1]$ for $|m|=|m_{0}|$, $m\ne m_{0}$.
Denote $\bm{p}=(p_{m})_{|m|=|m_{0}|}$.
For $\lambda\in [\lambda_{0}-1,\lambda_{0}+1]$, we define
\begin{equation}
    G_{\bm{p},\lambda}(v)=\Gamma_{P}F_{\lambda}(u_{0}+v),
\end{equation}
for $v\in \Gamma_{P}L^{2}(\mathbb{T}^{d+1})$. Thus, $G$ defines a map from $\Gamma_{P}L^{2}(\mathbb{T}^{d+1})$ to $\Gamma_{P}L^{2}(\mathbb{T}^{d+1})$. Moreover, it is straight forward to  check that if $v\in \Gamma_{P}L_{e}^{2}(\mathbb{T}^{d+1})$, then $G(v)\in\Gamma_{P}L_{e}^{2}(\mathbb{T}^{d+1})$. We aim to solve 
\begin{equation}
    G_{\bm{p}, \lambda}(v)=0.
    \label{qequationfanhan}
\end{equation}
 We call $v$ an $O(\delta)$-approximate solution to (\ref{qequationfanhan}) if we have
 \[
    \lVert G_{\bm{p},\lambda}(v) \rVert\le \delta.
 \]
Denote $I_{0}=[1,2]\times[-1,1]^{b-1}\times[\lambda-1,\lambda+1]$, where $b$ refers to the number of $m\in \mathbb{Z}^{d}$ such that $|m|=|m_{0}|$. For $(\bm{p},\lambda)\in I_{0}$,
we have
\begin{align}
    \lVert G_{\bm{p},\lambda}(0) \rVert&=\lVert \Gamma_{P} F_{\lambda}(u_{0}) \rVert \notag\\
    &=\lVert \varepsilon^{2} D^{\alpha}(u_{0}^{3}) \rVert\lesssim \varepsilon^{2}.
\end{align}
 For  $j_{0}=\frac{\log(\frac{1}{2}\log\frac{1}{\varepsilon})}{\log M^{c}}$, $0$ is an $O(e^{-2(M^{j_{0}})^{c}})$-approximate solution to (\ref{qequationfanhan}). Denote $v_{j_{0}}=0$, $u_{j_{0}}=u_{0}+v_{j_{0}}$ and $\Lambda_{j_{0}}=\{I_{0}\}$. 
Note that we have
\begin{itemize}
    \item $u_{j_0}\in L^{2}_{e}(\mathbb{T}^{d+1})$ and $u_{j_{0}}$ is defined smoothly for $(\bm{p},\lambda)\in I_{0}$. Moreover, $\hat{u}_{j_{0}}(m,n)=\hat{u}_{j_{0}}(-m,-n)\in\mathbb{R}$ for $(m,n)\in \mathbb{Z}^{d}$.
    \item $\textup{supp }\hat{u}_{j_{0}}\subset B(0,M^{j_{0}})$.
    \item $|\partial^{\beta} \hat{u}_{j_{0}}(\xi)|<C_{0} e^{-|\xi|_{1}^{c}}$ for $\xi\in\mathbb{Z}^{d+1}$ and $\beta=0,1$, where $\partial$ refers to the derivative to $\lambda$ or $\bm{p}$, and $C_{0}$ is a constant depending only on $m_{0}$, $d$ and $c$.
    \item For $(\bm{p},\lambda)\in I\in\Lambda_{j_{0}}$, we  have
    \begin{equation}
        \lVert \partial^{\beta} G_{\bm{p},\lambda}(v_{j_{0}}) \rVert=\lVert  \partial^{\beta} \Gamma_{P}F_{\lambda}(u_{j_{0}}) \rVert<e^{-2 (M^{j_{0}})^{c}},
    \end{equation}
    for $\beta=0, 1$.
\end{itemize}
 
Let $w\in \Gamma_{P} L^{2}(\mathbb{T}^{d})$, we have
\begin{align}
    G_{\bm{p},\lambda}(v_{j_{0}}+w)
    =&\Gamma_{P} F_{\lambda}(u_{j_{0}}+w)\notag\\
    =&\Gamma_{P}\left\{ (\lambda\partial_{\theta})^{2}(u_{j_{0}}+w)-\Delta(u_{j_{0}}+w)+\rho(u_{j_{0}}+w)+\varepsilon^{2} D^{\alpha}((u_{j_{0}}+w)^{3}) \right\} \notag\\
    =&\Gamma_{P}\{F_{\lambda}(u_{j_{0}})+((\lambda\partial_{\theta})^{2}-\Delta+\rho)w+\varepsilon^{2} D^{\alpha}(3u_{j_{0}}^{2}w)+\varepsilon^{2} D^{\alpha}(3uw^{2}+w^{3})\}\notag\\
    =&G_{\bm{p},\lambda}(v_{j_{0}})+\Gamma_{P}((\lambda\partial_{\theta})^{2}-\Delta+\rho)\Gamma_{P}w+\varepsilon^{2}\Gamma_{P}D^{\alpha}(3u_{j_{0}}^{2}\Gamma_{P}w)\notag\\
    &+\varepsilon^{2} \Gamma_{P}D^{\alpha}(3uw^{2}+w^{3}).\label{taylor}
\end{align}
We obtain the linearized equation 
\begin{equation}
    \Gamma_{P}((\lambda\partial_{\theta})^{2}-\Delta+\rho)\Gamma_{P}w+\varepsilon^{2} D^{\alpha}(3u_{j_{0}}^{2}\Gamma_{P}w)=-G_{\bm{p},\lambda}(v_{j_{0}}).
\end{equation}
Passing to the Fourier coefficients, we have
\begin{equation}
    R_{P}(D+\varepsilon^{2} \Lambda S_{j_{0}})R_{P}\hat{w}=-\widehat{G}_{\bm{p},\lambda}(v_{j_{0}}),
\end{equation}
where $R_{P}$ refers to coordinates restriction to $\mathbb{Z}^{d+1}\setminus S$,
\begin{equation}
    D=\textup{diag} (-(n\lambda)^{2}+\mu_{m}^{2}:(m,n)\in\mathbb{Z}^{d+1}),
\end{equation}
\begin{equation}
    \Lambda=\textup{diag}(n\cdot 0+\langle m \rangle^{\alpha}),
\end{equation}
and 
\begin{equation}
    S_{j_{0}}=S_{3u_{j_{0}}^{2}}.
\end{equation}
Here, $S_{\phi}$ represents the Toeplitz operator corresponding to $\phi$ (i.e., $S_{\phi}((m,n),(m',n'))=\hat{\phi}(m-m',n-n')$).
Denote $T_{j_{0}}=R_{P}(D+\varepsilon^{2} \Lambda S_{j_{0}})R_{P}$ and $T_{j_{0}, N}=T_{j_{0}}|_{|(m,n)|_{1}<N}$. Note that $T_{j_{0}}$ is not self-adjoint, while self-adjoint property is important in the analysis of the measure. Thus, we introduce the following matrix:

Let 
\begin{equation}
    \tilde{D}=\textup{diag}\left(\frac{-(n\lambda)^{2}+\mu_{m}^{2}}{n\cdot 0+\langle m\rangle^{\alpha}}:(m,n)\in\mathbb{Z}^{d+1}\right).
\end{equation}
Denote $\tilde{T}_{j_{0}}=R_{P}(\tilde{D}+ \varepsilon^{2} S_{j_{0}})R_{P}$ and $\tilde{T}_{j_{0},N}=\tilde{T}_{j_{0}}|_{|(m,n)|_{1}<N}$. Note that $\tilde{T}_{j_{0}}$ is self-adjoint and depends on $(\bm{p},\lambda)$.

Next, we construct $u_{j_{0}+1}, v_{j_{0}+1}$ and $\Lambda_{j_{0}+1}$.

Let $N=M^{j_{0+1}}$. Mesh  $\mathop{\cup}\limits_{I\in\Lambda_{j_{0}}} I$ into intervals $I'$ of size $\exp{(-(\log N)^{C_{1}})}$.
Suppose $(\bm{p}, \lambda)$ satisfies 
\begin{align}
    &\lVert \tilde{T}_{j_{0}, N}^{-1} \rVert < \exp(\log N)^{C_{2}} \label{l2bound2},\\
    &|\tilde{T}_{j_{0}, N}^{-1}(\xi,\xi')|< e^{-\frac{1}{2}|\xi-\xi'|_{1}^{c}} \textup{ for } |\xi-\xi'|_{1}>N^{\frac{1}{2}} \label{offdiagonal2}.
\end{align}
By perturbation argument, we may ensure that for $(\bm{p}', \lambda')$ in the neighborhood of $(\bm{p},\lambda)$ of size $3\sqrt{b+2}\exp{(-(\log N)^{C_{1}})}$, we have
\begin{align}
    &\lVert \tilde{T}_{j_{0}, N}^{-1} \rVert <2\exp(\log N)^{C_{2}} \label{l2bound2'},\\
    &|\tilde{T}_{j_{0}, N}^{-1}(\xi,\xi')|<2e^{-\frac{1}{2}|\xi-\xi'|_{1}^{c}} \textup{ for } |\xi-\xi'|_{1}>N^{\frac{1}{2}} \label{offdiagonal2'},
\end{align}
since $C_{1}>C_{2}$.

If $I'$ contains a point $(\bm{p},\lambda)$ such that (\ref{l2bound2}) and (\ref{offdiagonal2}) holding, then we collect these $I'$ together and denote this collection by $\Lambda_{j_{0+1}}$. 

For a set $A\subset \mathbb{R}^{b+1}$, we denote by $B(A,\delta)$ the $\delta$-neighborhood of $A$.
By the above argument, we obtain that for $(\bm{p}, \lambda)\in \mathop{\cup}\limits_{I\in\Lambda_{j_{0}+1}} B(I,\sqrt{b+2}\exp{(-(\log N)^{C_{1}})})$, we have (\ref{l2bound2'}) and (\ref{offdiagonal2'})  which imply
\begin{align}
    &\lVert T_{j_{0}, N}^{-1} \rVert <2\exp(\log N)^{C_{2}} \label{l2bound2''},\\
    &|T_{j_{0}, N}^{-1}(\xi,\xi')|<2e^{-\frac{1}{2}|\xi-\xi'|_{1}^{c}} \textup{ for } |\xi-\xi'|_{1}>N^{\frac{1}{2}} \label{offdiagonal2''}.
\end{align}

By the definition of $S_{j_{0}}$, we have 
\begin{align}
   |\partial^{\beta} S_{j_{0}}(\xi,\xi')|<C_{0}'e^{-(1-)|\xi-\xi'|_{1}^{c}},
\end{align}
for $\beta=0,1$ and $\xi,\xi'\in\mathbb{Z}^{d+1}$. Here, $C_{0}'$ only depends on $C_{0}$ and the decayed rate (i.e., the `$-$' of $1-$). 
Thus, we have
\begin{equation}
    |\partial^{\beta}T_{j_{0}}(\xi,\xi')|<\varepsilon^{\frac{9}{5}} \langle m \rangle^{\alpha} e^{-(1-)|\xi-\xi'|_{1}^{c}}, \textup{ for } \xi\ne\xi',  \xi=(m,n),
    \label{2627}
\end{equation}
where $\beta=0, 1$. 

From now on, we consider $(\bm{p}, \lambda)\in \mathop{\cup}\limits_{I\in\Lambda_{j_{0}+1}} B(I,\sqrt{b+2}\exp{(-(\log N)^{C_{1}})})$.
Since one has (\ref{l2bound2''}) and
\begin{equation}
    \partial T_{j_{0},N}^{-1}=-T_{j_{0}, N}^{-1}(\partial T_{j_{0}, N})T_{j_{0}, N}^{-1},
\end{equation}
it follows that 
\begin{align}
    \lVert \partial T_{j_{0}, N}^{-1} \rVert\le& \lVert T_{j_{0}, N}^{-1} \rVert^{2} \lVert \partial T_{j_{0},N} \rVert\notag\\
    \le & e^{3(\log N)^{C_{2}}}.
\end{align}
When $|\xi-\xi'|_{1}>N^{3/4}$, we have
\begin{align}
    |(\partial T_{j_{0}, N}^{-1})(\xi,\xi')|&\le \sum\limits_{|\xi_{1}|,|\xi_{2}|<N}\left| T_{j_{0}, N}^{-1}(\xi,\xi_{1})  \right| \left| \partial T_{j_{0}, N}(\xi_{1},\xi_{2}) \right| \left| T_{j_{0}, N}^{-1}(\xi_{2}, \xi') \right|\notag\\
        &\le N^{2(d+1)+3} e^{5(\log N)^{C_{2}}} e^{-\frac{1}{2}(|\xi-\xi'|_{1}-3 N^{1/2})^{c}}\notag\\
        &\le e^{-(\frac{1}{2}-)|\xi-\xi'|_{1}^{c}}.\label{32w}
\end{align}
Since $u_{j_{0}}\in L^{2}_{e}(\mathbb{T}^{d+1})$ and  $\textup{supp } \hat{u}_{j_{0}}\subset B(0, M^{j_{0}})$, we have
$$G_{\bm{p}, \lambda}(v_{j_{0}})\in L^{2}_{e}(\mathbb{T}^{d+1}),$$ and $$\textup{supp }\hat{G}_{\bm{p}, \lambda}(v_{j_{0}})\subset B(0, 10 M^{j_{0}}).$$
Let 
\begin{equation}
    \hat{w}_{j_{0}}=-T_{j_{0}, N}^{-1}\widehat{G}_{\bm{p},\lambda}(v_{j_{0}}),
    \label{bijinbu}
\end{equation}
and 
\begin{equation}
    v_{j_{0}+1}=v_{j_{0}}+w_{j_{0}}, \ \ u_{j_{0}+1}=u_{0}+v_{j_{0}+1}.
\end{equation}
It is straightforward that $\hat{w}_{j_{0}}(\xi)=\hat{w}_{j_{0}}(-\xi)\in\mathbb{R}$ and $\textup{supp } \hat{w}_{j_{0}}\subset B(0,\frac{1}{4}M^{j_{0}+1})$. By the definition (\ref{bijinbu}), we have
\begin{align}
    \lVert w_{j_{0}} \rVert\le& \lVert T_{j_{0}, N}^{-1} \rVert \lVert G_{\bm{p}, \lambda}(v_{j_{0}}) \rVert\notag\\
    \le & e^{(\log N)^{C_{2}}} e^{-2 (M^{j_{0}})^{c}}\notag\\
    \le & e^{-(2-)(M^{j_{0}})^{c}}<e^{\frac{3}{2} (M^{j_{0}+1})^{c}},
    \label{36w}
\end{align}
provided that $c<\frac{\log \frac{4}{3}}{\log M}$. Moreover, we have
\begin{align}
    \lVert \partial w_{j_{0}} \rVert\le&\lVert \partial T_{j_{0}, N}^{-1} \rVert \lVert G_{\bm{p}, \lambda}(v_{j_{0}}) \rVert+\lVert  T_{j_{0}, N}^{-1} \rVert \lVert \partial G_{\bm{p}, \lambda}(v_{j_{0}}) \rVert\notag\\
    \le & e^{3(\log N)^{C_{2}}} e^{-2(M^{j_{0}})^{c}}\notag\\
    <& e^{-(2-)(M^{j_{0}})^{c}}<e^{-\frac{3}{2}(M^{j_{0}+1})^{c}}.\label{37w}
\end{align}

By (\ref{taylor}), we have
\begin{align}
     G_{\bm{p},\lambda}(v_{j_{0}+1})=& \left((T_{j_{0}}-T_{j_{0},N})\hat{w}_{j_{0}}\right)^{\vee}+\varepsilon^{2} \Gamma_{P} D^{\alpha}(3u_{j_{0}}w_{j_{0}}^{2}+w_{j_{0}}^{3}).
\end{align}
For the second term, we have
\begin{equation}
    \lVert \partial^{\beta} \varepsilon^{2} \Gamma_{P} D^{\alpha}(3u_{j_{0}}w_{j_{0}}^{2}+w_{j_{0}}^{3}) \rVert<e^{-(3-)(M^{j_{0}+1})^{c}},
\end{equation}
for $\beta=0,1$.
Now we consider the first term. Denote $P_{K}$ the projection of the Fourier coefficients to $B(0, K)$. For simplicity, we omit the subscript of $T_{j_{0}, N}$. We have
\begin{align}
    (T-T_{N})\hat{w}_{j}
    &=(I-P_{N})T P_{\frac{N}{2}}\hat{w}_{j} +(T-T_{N})(\hat{w}_{j}-P_{\frac{N}{2}}\hat{w}_{j})\notag\\
    &=(I-P_{N})T P_{\frac{N}{2}}\hat{w}_{j}-(T-T_{N})(I-P_{\frac{N}{2}})T_{N}^{-1} \widehat{G}_{\bm{p},\lambda}(v_{j_{0}})\notag\\
    &=(I-P_{N})T P_{\frac{N}{2}}\hat{w}_{j}-(T-T_{N})(I-P_{\frac{N}{2}})T_{N}^{-1} P_{\frac{N}{4}}\widehat{G}_{\bm{p},\lambda}(v_{j_{0}}).
\end{align}

By (\ref{2627}), (\ref{36w}) and (\ref{37w}), we obtain
\begin{align}
    \lVert \partial^{\beta} (I-P_{N})T P_{\frac{N}{2}}\hat{w}_{j}  \rVert
    <& N^{\alpha} e^{-(1-)(\frac{1}{2}N)^{c}} e^{-\frac{3}{2}(M^{j_{0}+1})^{c}}\notag \\
    <&\frac{1}{3} e^{-\frac{1}{2} N^{c}} e^{-\frac{3}{2} (M^{j_{0}+1})^{c}}=\frac{1}{3}e^{-2 (M^{j_{0}+1})^{c}},
\end{align}
for $\beta=0,1$.
By (\ref{offdiagonal2''}), (\ref{32w}) and the estimate of $\partial^{\beta}G_{p,\lambda}(v_{j_{0}})$, we get
\begin{align}
    \lVert \partial^{\beta} (T-T_{N})(I-P_{\frac{N}{2}})T_{N}^{-1} P_{\frac{N}{4}}\widehat{G}_{\bm{p},\lambda}(v_{j_{0}}) \rVert
    &< e^{-(\frac{1}{2}-)(\frac{1}{4} N)^{c}} e^{-2(M^{j_{0}})^{c}}\notag\\
        &\le \frac{1}{3}e^{-\frac{17}{8} (M^{j_{0}})^{c}}\notag\\
        &\le \frac{1}{3}e^{-2 (M^{j_{0}+1})^{c}},
\end{align}
for $\beta=0,1$, provided that $c<\frac{\log \frac{17}{16}}{\log M}$.
Hence, we obtain
\begin{equation}
    \lVert \partial^{\beta} (T-T_{N})\hat{w}_{j_{0}} \rVert<\frac{2}{3} e^{-2(M^{j_{0}+1})^{c}},
\end{equation}
for $\beta=0,1$. Moreover, we have
\begin{equation}
    \lVert \partial^{\beta} G_{\bm{p}, \lambda}(v_{j_{0}+1}) \rVert < e^{-2(M^{j_{0}+1})^{c}},
\end{equation}
for $\beta=0,1$.
Note that in the above $w_{j_{0}}$ is defined on $\mathop{\cup}\limits_{I\in\Lambda_{j_{0}+1}} B(I,\sqrt{b+2}\exp{(-(\log N)^{C_{1}})})$.  We need to extend its definition to the entire $(\bm{p},\lambda)$-parameter set $I_{0}$.

Define the following smooth function defined on $\mathbb{R}^{b+1}$
\begin{equation}
    \varphi(x)=
    \begin{cases}
        e^{\frac{1}{| x|^{2}-1}}, \ | x|<1,\\
        0, \ \ |x|\ge 1.
    \end{cases}
\end{equation}
Let $\psi(x)=\frac{1}{C} \varphi(x)$, where $C=\int_{x\in\mathbb{R}^{b+1}} \varphi(x)dx$. Define $\psi_{\delta}(x)=\frac{1}{\delta^{b+1}}\psi(\frac{x}{\delta})$. Let $\chi_{A}$ be characteristic function of $A$.
Denote $A_{j_{0}+1}=\mathop{\cup}\limits_{I\in\Lambda_{j_{0}+1}} B(I,\frac{1}{2}\sqrt{b+2}\exp{(-(\log N)^{C_{1}})})$ and $\delta_{j_{0}+1}=\frac{1}{2}\sqrt{b+2}\exp{(-(\log N)^{C_{1}})}$.
Let $\phi_{j_{0}+1}=\chi_{A_{j_{0}+1}}*\psi_{\delta_{j_{0}+1}}$. Thus, we have
\begin{itemize}
    \item $\phi_{j_{0}+1}(\bm{p},\lambda)=1$ when $(\bm{p},\lambda)\in \mathop{\cup}\limits_{I\in\Lambda_{j_{0}+1}} I$.
    \item $|\partial \phi_{j_{0}+1}|\lesssim\exp{(\log N)^{C_{1}}}$.
\end{itemize}
Let $\tilde{w}_{j_{0}}=w_{j_{0}}\phi_{j_{0}+1}$. Then, $\tilde{w}_{j_{0}}$ is defined smoothly on $I_{0}$. Moreover, we have that, for $(\bm{p},\lambda)\in I_{0}$,
\begin{align}
    &\tilde{w}_{j_{0}}\in L_{e}^{2}(\mathbb{T}^{d+1}),\\
    &\textup{supp } \hat{\tilde{w}}_{j_{0}}\subset B(0, M^{j_{0}+1}),\\
    &\lVert \partial^{\beta}  \tilde{w}_{j_{0}} \rVert<e^{-(\frac{3}{2}-)(M^{j_{0}+1})^{c}}.
\end{align}
Let
\begin{equation}
    v_{j_{0}+1}=v_{j_{0}}+\tilde{w}_{j_{0}}, \ \ u_{j_{0}+1}=u_{0}+v_{j_{0}+1}.
\end{equation}
We conclude that we have $v_{j_{0}+1}$, $u_{j_{0}+1}$ and $\Lambda_{j_{0}+1}$ such that
\begin{itemize}
    \item $u_{j_{0}+1}\in L^{2}_{e}(\mathbb{T}^{d+1})$ and $u_{j_{0}+1}$ is defined smoothly for $(\bm{p},\lambda)\in I_{0}$. Moreover, $\hat{u}_{j_{0}+1}(m,n)=\hat{u}_{j_{0}+1}(-m,-n)\in\mathbb{R}$ for $(m,n)\in \mathbb{Z}^{d}$.
    \item $\textup{supp }\hat{u}_{j_{0}+1}\subset B(0,M^{j_{0}})$.
    \item $|\partial^{\beta} \hat{u}_{j_{0}+1}(\xi)|<(C_{0}+\sum\limits_{j'=j_{0}}^{j_{0}+1}\frac{1}{j'^{2}}) e^{-|\xi|_{1}^{c}}$ for $\xi\in\mathbb{Z}^{d+1}$ and $\beta=0,1$, where $\partial$ refers to the derivative to $\lambda$ or $\bm{p}$.
    \item For $(\bm{p},\lambda)\in I\in\Lambda_{j_{0}+1}$, we  have
    \begin{equation}
        \lVert \partial^{\beta} G_{\bm{p},\lambda}(v_{j_{0}+1}) \rVert=\lVert  \partial^{\beta} \Gamma_{P}F_{\lambda}(u_{j_{0}+1}) \rVert<e^{-2 (M^{j_{0}+1})^{c}},
    \end{equation}
    for $\beta=0, 1$.
\end{itemize}
Iterating the above process, we get the following statement:
\begin{proposition}
    For $j\ge j_{0}$, there exists $v_{j}$ and $\Lambda_{j}$ satisfying the following properties:
    \begin{itemize}
        \item[(j.1)] $u_{j}=u_{0}+v_{j}\in L^{2}_{e}(\mathbb{T}^{d+1})$ and $u_{j}$ is defined smoothly for $(\bm{p},\lambda)\in I_{0}$. Moreover, $\hat{u}_{j}(m,n)=\hat{u}_{j}(-m,-n)\in\mathbb{R}$ for $(m,n)\in \mathbb{Z}^{d}$.
        \item[(j.2)] $\textup{supp }\hat{u}_{j}\subset B(0,M^{j})$.
        \item[(j.3)] $|\partial^{\beta} \hat{u}_{j}(\xi)|<(C_{0}+\sum\limits_{j'= j_{0}}^{j}\frac{1}{j'^{2}}) e^{-|\xi|_{1}^{c}}$ for $\xi\in\mathbb{Z}^{d+1}$ and $\beta=0,1$, where $\partial$ refers to the derivative to $\lambda$ or $\bm{p}$, and $C_{0}$ is a constant depending only on $m_{0}$, $d$ and $c$.
        \item[(j.4)] $\lVert \partial^{\beta}(u_{j}-u_{j-1}) \rVert<e^{-(\frac{3}{2}-)(M^{j})^{c}}$.
        \item[(j.5)] Let $N_{j}=M^{j}$. $\Lambda_{j}$ ($j>j_{0}$) is a collection of disjoint intervals of size $\exp (-(\log N_{j})^{C_{1}})$ satisfying:
        \begin{itemize}
            \item[(j.5.a)] For any $I'\in\Lambda_{j}$, there exists $I\in\Lambda_{j-1}$ such that $I'\subset I$.
            \item[(j.5.b)] For $(\bm{p},\lambda)\in I\in\Lambda_{j}$, we  have
                \begin{equation}
                    \lVert \partial^{\beta} G_{\bm{p},\lambda}(v_{j}) \rVert=\lVert  \partial^{\beta} \Gamma_{P}F_{\lambda}(u_{j}) \rVert<e^{-2 (M^{j})^{c}},
                \end{equation}
            for $\beta=0, 1$.
            \item[(j.5.c)]  For $(\bm{p},\lambda)\in I\in\Lambda_{j}$, we  have
            \begin{align}
               &\lVert \tilde{T}_{j-1, N_{j}}^{-1}\rVert<2\exp (\log N_{j})^{C_{2}},\\
               &|\tilde{T}_{j-1, N_{j}}^{-1}(\xi,\xi')|<2e^{-\frac{1}{2}|\xi-\xi'|_{1}^{c}}, \textup{ for } |\xi-\xi'|_{1}>N^{\frac{1}{2}}.
            \end{align}
            Here,   $\tilde{T}_{j-1}=\tilde{D}+\varepsilon^{2} S_{3  u_{j-1}^{2}}$. Moreover, if a point $(\bm{p},\lambda)\in I\in\Lambda_{j-1}$ ensures that
            \begin{align}
                &\lVert \tilde{T}_{j-1, N_{j}}^{-1}\rVert<\exp (\log N_{j})^{C_{2}},\label{tiaojian1tocheck}\\
               &|\tilde{T}_{j-1, N_{j}}^{-1}(\xi,\xi')|<e^{-\frac{1}{2}|\xi-\xi'|_{1}^{c}}, \textup{ for } |\xi-\xi'|_{1}>N^{\frac{1}{2}}\label{tiaojian2tocheck},
            \end{align}
            then there exists an interval $I'\in \Lambda_{j}$ such that $(\bm{p},\lambda)\in I'$.
        \end{itemize}
    \end{itemize}
    \label{qiterative}
\end{proposition}
\begin{remark}
    Note that $\Lambda_{j}$ is constructive and we do not yet have an estimate for its measure.
\end{remark}

Let $v=\lim\limits_{j\to \infty}v_{j}$. Then $v\in L_{e}^{2}(\mathbb{T}^{d+1})$ is defined smoothly on $I_{0}$. For $(\bm{p},\lambda)\in\mathop{\cap}\limits_{j\ge j_{0}}\mathop{\cup}\limits_{I\in\Lambda_{j}}I$, $v$ solves 
\begin{equation}
    G_{\bm{p},\lambda}(v)=0.
\end{equation}
Let $u=u_{0}+v$. Then, $u\in  L_{e}^{2}(\mathbb{T}^{d+1})$ satisfies that
\begin{align}
    &\hat{u}(m_{0},1)=\hat{u}(-m_{0},-1)=\frac{1}{2} p_{0},\\
    &\hat{u}(m,1)=\hat{u}(-m,-1)=\frac{1}{2}p_{m}, \textup{ for } |m|=|m_{0}|, m\ne m_{0},\\
    &\sum\limits_{\substack{(m,n)\in\mathbb{Z}^{d+1},\\(m,n)\notin S}} |\hat{u}(m,n)|e^{|(m,n)|_{1}^{c}}<e^{-\frac{1}{4}(M^{j_{0}})^{c}}<\varepsilon^{\frac{1}{8}}.
\end{align}
Moreover, we have
\begin{equation}
    |\partial^{\beta} u|\lesssim 1.
\end{equation}

\section{Solving the Q-equations and estimating the measure}
In this section, we solve the $Q$-equations and estimate the measure.

\subsection{Solving the Q-equations}
Now we come back to the $Q$-equations:
\begin{equation}
     (-(n\lambda)^{2}+\mu_{m}^{2})\hat{u}(m,n)+\varepsilon^{2}\langle m \rangle^{\alpha}(u^{3})^{\land}(m,n)=0,\  (m,n)\in S.
\end{equation}
More precisely, we have the following $Q$-equations:
\begin{equation}
    \frac{1}{2}(-\lambda^{2}+|m_{0}|^{2}+\rho)p_{m}+\varepsilon^{2}\langle m_{0} \rangle^{\alpha} (u^{3})^{\land}(m,1)=0,\ \ |m|=|m_{0}|. 
    \label{qequapre}
\end{equation}

Denote 
\begin{equation}
    \sigma=\varepsilon^{-2}(-\lambda^{2}+|m_{0}|^{2}+\rho)
    \label{sigmadef}
\end{equation}
and rewrite (\ref{qequapre}) as 
\begin{equation}
    \sigma p_{m}+2\langle m_{0} \rangle^{\alpha} (u^{3})^{\land}(m,1)=0,\ \ |m|=|m_{0}|.
    \label{bianxingaq}
\end{equation}

In particular, one has 
\begin{equation}
    u(x,\theta)=u_{0}(x,\theta)+O(\varepsilon^{\frac{1}{8}}),
\end{equation}
where $u_{0}(x,\theta)=p_{0}\cos( m_{0}\cdot x  +\theta)+\sum\limits_{\substack{|m|=|m_{0}|\\ m\ne m_{0}}} p_{m} \cos( m\cdot x+\theta)$. Denote $p_{m_{0}}=p_{0}$. Thus, we have
\begin{align}
\widehat{u^{3}}(m,1)=&\widehat{u_{0}^{3}}(m,1)+O(\varepsilon^{\frac{1}{8}})\\
=&\frac{1}{(2\pi)^{d+1}}\int_{\mathbb{T}^{d+1}}\left[  \sum\limits_{|m'|=|m_{0}|} p_{m'} \cos( m'\cdot x+\theta) \right]^{3}\cos(m\cdot x+\theta)dxd\theta\label{termtocalculate}\\
&+O(\varepsilon^{\frac{1}{8}}).\notag
\end{align}

Write
\begin{align}
    &\left[  \sum\limits_{|m'|=|m_{0}|} p_{m'} \cos( m'\cdot x+\theta) \right]^{3}\\
    =&p_{0}^{3} \cos^{3}(m_{0}\cdot x+\theta)\\
    +&3p_{0}^{2}\cos^{2}(m_{0}\cdot x+\theta)\left[\sum\limits_{\substack{|m'|=|m_{0}|\\ m'\ne m_{0}}} p_{m'} \cos( m'\cdot x+\theta)\right]\\
    +&O(|p_{m'}|^{2}; m'\ne m_{0}).
\end{align}
Now we calculate (\ref{termtocalculate}).

\textbf{Case $\bm{m=m_{0}}$.}
Note that $\cos^{4} \alpha=\frac{1}{8}\cos 4\alpha+\frac{1}{2}\cos 2\alpha+\frac{3}{8}$ and $\cos^{3}\alpha=\frac{1}{4}\cos 3\alpha+\frac{3}{4} \cos \alpha$. We have
\begin{align}
    &(\ref{termtocalculate})\notag\\
    =&\frac{1}{(2\pi)^{d+1}} \int_{\mathbb{T}^{d+1}} p_{0}^{3} \cos^{4}(m_{0}\cdot x+\theta)dxd\theta\notag\\
    &+3p_{0}^{2}\cos^{3}(m_{0}\cdot x+\theta)\left[ \sum\limits_{\substack{|m'|=|m_{0}|\\ m'\ne m_{0}}} p_{m'} \cos( m'\cdot x+\theta) \right]dxd\theta\notag\\
    &+O(|p_{m'}|^{2}; m'\ne m_{0})\notag\\
    =&\frac{1}{(2\pi)^{d+1}} \int_{\mathbb{T}^{d+1}} p_{0}^{3}\left(  \frac{1}{8}\cos(4m_{0}\cdot x+4\theta)+\frac{1}{2}\cos(2m_{0}\cdot x+2\theta)+\frac{3}{8} \right)\\
    +& 3p_{0}^{3}\left( \frac{1}{4}\cos(3m_{0}\cdot x+3\theta)+\frac{3}{4}\cos(m_{0}\cdot x+\theta) \right)
    \left[ \sum\limits_{\substack{|m'|=|m_{0}|\\ m'\ne m_{0}}} p_{m'} \cos( m'\cdot x+\theta) \right]dxd\theta\notag\notag\\
    &+O(|p_{m'}|^{2}; m'\ne m_{0})\notag\\
    =&\frac{3}{8}p_{0}^{3}+O(|p_{m'}|^{2}; m'\ne m_{0}).
\end{align}
Thus, for $m=m_{0}$, equation (\ref{bianxingaq}) becomes
\begin{equation}
    \sigma p_{0}+\frac{3}{4}\langle m_{0} \rangle^{\alpha} p_{0}^{3}+O(|p_{m'}|^{2}; m'\ne m_{0})+O(\varepsilon^{\frac{1}{8}})=0.
\end{equation}

\textbf{Case $\bm{m\ne m_{0}}$.}
Note that for $|m_{0}|=|m|=|m'|$, we have
\begin{equation}
    \int_{\mathbb{T}^{d+1}} \cos(2m_{0}\cdot x+2\theta)\cos(m\cdot x+\theta)\cos(m'\cdot x+\theta)=0,
\end{equation}
unless $m=m'=m_{0}$. Thus, we have
\begin{align}
    &(\ref{termtocalculate})\notag\\
    =&\frac{1}{(2\pi)^{d+1}}\int_{\mathbb{T}^{d+1}} p_{0}^{3}\cos^{4}(m_{0}\cdot x+\theta)\cos(m\cdot x+\theta)\notag\\
    &+3p_{0}^{2}\cos^{2}(m_{0}\cdot x+\theta)\cos(m\cdot x+\theta)\left[ \sum\limits_{\substack{|m'|=|m_{0}|\\ m'\ne m_{0}}} p_{m'} \cos( m'\cdot x+\theta) \right]dxd\theta\notag\\
    &+O(|p_{m'}|^{2}; m'\ne m_{0})\notag\\
    =&\frac{1}{(2\pi)^{d+1}} \int_{\mathbb{T}^{d+1}} p_{0}\left(\frac{1}{4}\cos (3m_{0}\cdot x+3\theta)+\frac{3}{4}\cos(m_{0}\cdot x+\theta)\right)\cos(m\cdot x+\theta)\notag\\
    &+3p_{0}^{2}\frac{\cos(2m_{0}\cdot x+2\theta)+1}{2} \cos(m\cdot x+\theta)\left[ \sum\limits_{\substack{|m'|=|m_{0}|\\ m'\ne m_{0}}} p_{m'} \cos( m'\cdot x+\theta) \right]dxd\theta\notag\\
    &+O(|p_{m'}|^{2}; m'\ne m_{0})\notag\\
    =&\frac{3}{4} p_{0}^{2}p_{m}+O(|p_{m'}|^{2}; m'\ne m_{0}).
\end{align}
Thus, for $m=m_{0}$, equation (\ref{bianxingaq}) becomes
\begin{equation}
    \sigma p_{m}+\frac{3}{2}\langle m_{0} \rangle^{\alpha} p_{0}^{2}p_{m}+O(|p_{m'}|^{2}; m'\ne m_{0})+O(\varepsilon^{\frac{1}{8}})=0.
\end{equation}
Therefore, we have the $Q$-equations as follows:
\begin{align}
\sigma p_{0}+\frac{3}{4}\langle m_{0} \rangle^{\alpha} p_{0}^{3}+O(|p_{m'}|^{2}; m'\ne m_{0})+O(\varepsilon^{\frac{1}{8}})=0,\\
 \sigma p_{m}+\frac{3}{2}\langle m_{0} \rangle^{\alpha} p_{0}^{2}p_{m}+O(|p_{m'}|^{2}; m'\ne m_{0})+O(\varepsilon^{\frac{1}{8}})=0.
\end{align}
This implies 
\begin{align}
    \sigma=-\frac{3}{4} \langle m_{0} \rangle^{\alpha} p_{0}^{2}+O(\varepsilon^{\frac{1}{8}}),
    p_{m}=O(\varepsilon^{\frac{1}{8}}), \textup{ for } m\ne m_{0}.
\end{align}
Hence, by (\ref{sigmadef}), we obtain
\begin{equation}
    \lambda^{2}=|m_{0}|^{2}+\rho+\frac{3}{4}\langle m_{0} \rangle^{\alpha} p_{0}^{2}\varepsilon^{2}+O(\varepsilon^{\frac{17}{8}}).
    \label{tutua}
\end{equation}
Denote $\Gamma=\{(\bm{p}(p_{0}), \lambda(p_{0}))|p_{0}\in[1,2]\}$. It is left to estimate the measure of  $$\Pi_{0}\left(\Gamma\cap \left(I_{0}\setminus \left(\mathop{\cap}\limits_{j\ge j_{0}}\mathop{\cup}\limits_{I\in\Lambda_{j}}I \right)\right)\right),$$ where $\Pi_{0}$ refers to the projection on the $p_{0}$ variable.
\subsection{Measure estimate in finite step}
Recall $\tilde{T}_{j}=R_{P}(\tilde{D}+\varepsilon^{2} S_{3u_{j}^{2}})R_{P}$, where
\begin{equation}
    \tilde{D}=\textup{diag}\left(\frac{-(n\lambda)^{2}+\mu_{m}^{2}}{n\cdot 0+\langle m\rangle^{\alpha}}:(m,n)\in\mathbb{Z}^{d+1}\right).
\end{equation}
In the sequel of this paper, we omit $R_{P}$ for simplicity.
By (\ref{tutua}), we have
\begin{align}
    \tilde{D}_{m,n}&=\frac{-(n\lambda)^{2}+\mu_{m}^{2}}{n\cdot 0+\langle m\rangle^{\alpha}}\notag\\
    &=\frac{(1-n^{2})\rho-n^{2}|m_{0}|^{2}+|m|^{2}-\frac{3}{4}\varepsilon^{2}n^{2}\langle m_{0}\rangle^{2}p_{0}^{2}+O(n^{2}\varepsilon^{\frac{17}{8}})}{\langle m \rangle^{\alpha}}\notag\\
    &=\frac{(1-n^{2})\rho-n^{2}|m_{0}|^{2}+|m|^{2}+O(n^{2}\varepsilon^{2})}{\langle m \rangle^{\alpha}}\notag\\
    &=\frac{(1-n^{2})\rho-n^{2}|m_{0}|^{2}+|m|^{2}}{\langle m \rangle^{\alpha}}+O(n^{2}\varepsilon^{2}),
\end{align}
for $(\bm{p},\lambda)\in \Gamma$.
Since $\rho$ satisfies
\begin{equation}
    |n\rho-k|>\gamma|n|^{-C_{3}}, \textup{ for } n,k\in\mathbb{Z},  n\ne 0,
\end{equation}
where $C_{3}=2d$.
Let 
\begin{equation}
N_{j+1}^{2}\varepsilon^{2}< N_{j+1}^{-2C_{3}-2}.
\label{xianzhi1}
\end{equation}
Then, for $|n|<N_{j+1}$, $|m|<N_{j+1}$, and $(m,n)\notin S$,  we have
\begin{equation}
    |\tilde{D}_{m,n}|>\gamma N_{j+1}^{-2 C_{3}-\alpha}-O(N_{j+1}^{2}\varepsilon^{2})>\frac{\gamma}{2}N_{j+1}^{-2 C_{3}-\alpha}.
    \label{xiaofenbujie}
\end{equation}
Denote $N=N_{j+1}$.
Note that we have
\begin{align}
    \tilde{T}_{j,N}^{-1}=&(I+\varepsilon^{2}\tilde{D}_{N}^{-1}S_{3u_{j}^{2}})^{-1}\tilde{D}_{N}^{-1}\\
    =&\tilde{D}_{N}^{-1}+\sum\limits_{k=1}^{\infty}(-\varepsilon^{2}\tilde{D}_{N}^{-1}S_{3u_{j}^{2}})^{k}\tilde{D}_{N}^{-1}.
\end{align}
By (\ref{xiaofenbujie}), we have
\begin{equation}
    \lVert \tilde{D}_{N}^{-1} \rVert<\frac{2}{\gamma} N^{2C_{3}+\alpha}.
\end{equation}
Thus, we have
\begin{equation}
    \lVert \tilde{T}_{j,N}^{-1} \rVert<2  \lVert \tilde{D}_{N}^{-1} \rVert<\frac{4}{\gamma} N^{2C_{3}+\alpha},
    \label{boundyouxianbu}
\end{equation}
as long as 
\begin{equation}
    \frac{2}{\gamma} C_{0}\varepsilon^{2} N^{2C_{3}+\alpha} <\frac{1}{10}.
    \label{xianzhi2}
\end{equation}
Besides, for $\xi,\xi'\in\mathbb{Z}^{d+1}$, $\xi\ne\xi'$, we have
\begin{align}
    |\tilde{T}_{j, N}^{-1}(\xi,\xi')|\le &\sum\limits_{k=1}^{\infty}|(-\varepsilon^{2}\tilde{D}_{N}^{-1} S_{3u_{j}^{2}})^{k}\tilde{D}_{N}^{-1}(\xi,\xi')| \notag\\
    \le & \sum\limits_{k=1}^{\infty} \varepsilon^{2k} (\frac{2}{\gamma}N^{2C_{3}+\alpha})^{k+1} \sum\limits_{|\xi_{1}|,...,|\xi_{k-1}|<N} e^{-(1-)(|\xi-\xi_{1}|_{1}^{c}+\cdots +|\xi_{k-1}-\xi'|_{1}^{c})}\notag\\
    \le & \sum\limits_{k=1}^{\infty} \varepsilon^{2k} (\frac{2}{\gamma}N^{2C_{3}+\alpha})^{k+1}(N^{d+1})^{k-1} e^{-(1-)|\xi-\xi'|_{1}^{c}}\notag\\
    < & e^{-(1-)|\xi-\xi'|_{1}^{c}} \sum\limits_{k=1}^{\infty} (\varepsilon^{2} N^{4 C_{3}+2\alpha+2d})^{k}\notag\\
    <& e^{-(1-)|\xi-\xi'|_{1}^{c}},\label{shuaijianyouxianbu}
\end{align}
as long as 
\begin{equation}
    \varepsilon^{2} N^{4C_{3}+2\alpha+2d}<\frac{1}{2}.
    \label{xianzhi3}
\end{equation}
Thus, for $\varepsilon^{2} N_{j+1}^{4C_{3}+2\alpha+5d}<1$, we have (\ref{boundyouxianbu}) and (\ref{shuaijianyouxianbu}), which imply (\ref{tiaojian1tocheck}) and (\ref{tiaojian2tocheck}).
Thus, we conclude that for
\begin{equation}
    \varepsilon^{2} N_{j+1}^{4C_{3}+2\alpha+5d}<1,
    \label{youxianbutiaojian}
\end{equation}
we have
\begin{equation}
    \Pi_{0}\left(\Gamma\cap \left(I_{0}\setminus \left(\mathop{\cup}\limits_{I\in\Lambda_{j+1}}I \right)\right)\right)=\emptyset.
\end{equation}

\subsection{The separation lemma}
 To estimate
 the measure of $\Gamma \cap \{(\mathop{\cup}\limits_{I\in\Lambda_{j}}I) \setminus (\mathop{\cup}\limits_{I'\in\Lambda_{j+1}}I')\}$
 for $\varepsilon^{2} N_{j}^{4C_{3}+2\alpha+5d}\ge 1$, 
 we need  the following arithmetical lemma which is an extension of Lemma 20.14 in \cite{bourgain2005green} to separate the singular sites in unbounded case:

 \begin{lemma}
     Let $B$ be a large number and $\mathcal{K}$ be a compact set.
     Choose constants $C$, $C'$, (depending on $\mathcal{K}$, $\tilde{b}$, $d$) and $C''$ (depending on $C$, $C'$, $\mathcal{K}$, $\tilde{b}$, $d$) large enough.
     Assume $\lambda'\in \mathcal{K}\subset\mathbb{R}^{\tilde{b}}$ satisfies the condition
     \begin{equation}
         |P(\lambda')|>B^{-C'}
         \label{guanyvlambdadejiashenlw}
     \end{equation}
     for all polynomials $P(X)\in \mathbb{Z}[X_{1},...,X_{\tilde{b}}]$, $P(X)\ne 0$ of degree less than $10 d$ and with coefficients $|a_{\beta}|<B^{C}$. Consider a sequence $(\xi_{j})_{1\le j\le k}$ of distinct elements of $\mathbb{Z}^{\tilde{b}+d}$ such that, for some $\sigma\in\mathbb{R}$, for all $j$, we have
     \begin{equation}
         |(\lambda'\cdot k_{j}+\sigma)^{2}-|n_{j}|^{2}|<B, \textup{ where } \xi_{j}=(n_{j}, k_{j})\
         \label{modification}
     \end{equation}
     and 
     \begin{equation}
         |\xi_{j}-\xi_{j-1}|<B.
     \end{equation}
Furthermore, assume that
\begin{equation}
    \max\limits_{n}(\#\{1\le j\le k| n_{j}=n\})<B'
    \label{duiyingndegeshuk}
\end{equation}
Then, we have
\begin{equation}
    k<(B B')^{C''}.
\end{equation}
\label{fenlixingdinglifornlw}
 \end{lemma}
The proof of this lemma is essentially the same as Lemma 20.14 in \cite{bourgain2005green}, except that the difference in (\ref{modification}) leads to some minor modifications. For the convenience of the reader, we provide the proof in Appendix \ref{proof}.

Having Lemma \ref{fenlixingdinglifornlw},
it is natural to introduce the following conditions:
\begin{definition}
    In this paper, we say a vector $\lambda\in  \mathbb{R}^{\tilde{b}}$ satisfies a generalized Diophantine condition, if for any polynomial $P(X)\in \mathbb{Z}[X_{1},....,X_{\tilde{b}}]$, $P(X)\ne 0$ of degree less than $d$, we have
    \begin{equation}
        |P(\lambda)|>\gamma \left(|\bm{a}_{P}|\right)^{-\tau}, 
        \notag
    \end{equation}
    where $\bm{a}_{P}$ is the vector made up of coefficients of the polynomial $P(X)$. We  denote this condition by $gDC_{\tilde{b}, d, \gamma, \tau}$. We also denote the set of vectors satisfying this condition by $gDC_{\tilde{b}, d, \gamma, \tau}$. When $\tilde{b}=1$, we omit $\tilde{b}$, i.e., $gDC_{d,\gamma,\tau}$.
    
    Furthermore, if $d$, $\gamma$ and $\tau$ is fixed, we denote by  $gDC_{M}$ as the condition:
    \begin{equation}
         |P(\lambda)|>\gamma \left(|\bm{a}_{P}|\right)^{-\tau}, 
        \notag
    \end{equation}
    for all polynomials $P(X)\in\mathbb{Z}[X_{1}, ..., X_{b}]$, $P(X)\ne 0$  of degree less than $d$ and with coefficients $|a_{\beta}|\le M$.
\end{definition}
\begin{remark}
    In this paper, we only use the case when $\lambda$ is a scale. However, we give the vector case for generality. 
\end{remark}
 See Appendix \ref{gDCproperties} for some facts about the generalized Diophantine conditions.

\subsection{Measure estimate for infinite steps}

Since  we have (\ref{NONRE2}), we obtain
\begin{equation}
    \left| \sum\limits_{k=0}^{10d} a_{k}(\sqrt{\rho+|m_{0}|^{2}})^{k} \right|>(\log\log\frac{1}{\varepsilon})^{-1}\left(\sum |a_{k}|\right)^{-\tilde{C}}, \textup{ for all } \{a_{k}\}\in\mathbb{Z}^{10d+1}\setminus\{0\}.
\end{equation}
By (\ref{tutua}), we have
\begin{align}
    \left| \sum\limits_{k=0}^{10d} a_{k}\lambda^{k} \right|>&(\log\log \frac{1}{\varepsilon})^{-1}\left(\sum |a_{k}|\right)^{-\tilde{C}}-10d\langle m_{0} \rangle^{\alpha+10d}\varepsilon \left(\sum |a_{k}|\right)\\
    >&\frac{1}{2}(\log\log \frac{1}{\varepsilon})^{-1}\left(\sum |a_{k}|\right)^{-\tilde{C}},
\end{align}
as long as $10d\langle m_{0} \rangle^{\alpha+10d}\varepsilon <\frac{1}{2}(\log\log \frac{1}{\varepsilon})^{-1}\left(\sum |a_{k}|\right)^{-\tilde{C}-1}$.
Let $C_{5}=4(\tilde{C}+1)d$ and
\begin{equation}
    R_{\bm{a}}=\left\{\lambda\in[\lambda_{0}-1, \lambda_{0}+1]: \left|\sum\limits_{k=0}^{10d} a_{k}\lambda^{k}\right|\le \left(\sum |a_{k}|\right)^{-C_{5}} \right\}.
\end{equation}
Then, by Lemma \ref{ekdegaojiedaoshuceduguji}, we have
\begin{equation}
    \textup{mes } R_{\bm{a}}<\left(\sum |a_{k}|\right)^{-\frac{C_{5}}{d}}.
\end{equation}
Let
\begin{equation}
    \tilde{R}_{\bm{a}}=\left\{p_{0}\in[1,2]: \left|\sum\limits_{k=0}^{10d} a_{k}\lambda^{k}(p_{0})\right|\le \left(\sum |a_{k}|\right)^{-C_{5}} \right\}.
\end{equation}
By (\ref{tutua}), we have
\begin{equation}
    \textup{mes }  \tilde{R}_{\bm{a}}<\frac{\left(\sum |a_{k}|\right)^{-C_{5}}}{\langle m_{0}\rangle^{\alpha}\varepsilon^{2}}<\left(\sum |a_{k}|\right)^{-\tilde{C}-1},
\end{equation}
for $10d\langle m_{0} \rangle^{\alpha+10d}\varepsilon \ge \frac{1}{2}(\log\log \frac{1}{\varepsilon})^{-1}\left(\sum |a_{k}|\right)^{-\tilde{C}-1}$.
Thus, we have
\begin{equation}
    \textup{mes } \mathop{\cup}\limits_{|\bm{a}|\ge \varepsilon^{-\frac{1}{2(\tilde{C}+1)}}} \tilde{R}_{\bm{a}}<\sum\limits_{|\bm{a}|\ge \varepsilon^{-\frac{1}{2(\tilde{C}+1)}}} |\bm{a}|^{-\tilde{C}-1}<\varepsilon^{\tilde{c}},
\end{equation}
where $\tilde{c}$ is a small constant depending on $d$.
Denote 
\begin{equation}
    W=\left\{ (\bm{p},\lambda): \left|\sum\limits_{k=0}^{10d} a_{k}\lambda^{k}\right|> (\log\log\frac{1}{\varepsilon})^{-1} \left(\sum |a_{k}|\right)^{-C_{5}} \textup{ for all } \{a_{j}\}\in\mathbb{Z}^{10d+1}\setminus\{0\}    \right\}.
\end{equation}
By the above argument, we have
\begin{equation}
    \textup{mes }\Pi_{0} (\Gamma\cap(I_{0}\setminus W))<\varepsilon^{\tilde{c}}.
\end{equation}

For simplicity, we denote 
\begin{equation}
    \tilde{\Lambda}_{j}=\mathop{\cup}\limits_{I\in\lambda_{j}} I
\end{equation}
and
\begin{equation}
    \tilde{\Lambda}^{(\infty)}=\mathop{\cap}_{j\ge j_{0}}\tilde{\Lambda}_{j}.
\end{equation}
We have 
\begin{align}
    \Gamma\cap\left(I_{0}\setminus\tilde{\Lambda}^{(\infty)}\right)
    \subset & \left[\Gamma\cap(I_{0}\setminus W)\right]\cup \left[\Gamma\cap(I_{0}\setminus\tilde{\Lambda}^{(\infty)})\cap W\right]\\
    \subset &  \left[\Gamma\cap(I_{0}\setminus W)\right]\cup \left\{\Gamma\cap\left[ \left(\mathop{\cup}\limits_{j\ge j_{0}}(\tilde{\Lambda}_{j}\setminus\tilde{\Lambda}_{j+1})\right)\cap W \right]      \right\}\\
    = &\left[\Gamma\cap(I_{0}\setminus W)\right]\cup\left\{ 
    \mathop{\cup}\limits_{j\ge j_{0}} \Gamma\cap \left[ (\tilde{\Lambda}_{j}\setminus\tilde{\Lambda}_{j+1})\cap W \right] \right\}.
\end{align}
Note that we have estimate $\Gamma\cap(I_{0}\setminus W)$ and $\Gamma\cap (\tilde{\Lambda}_{j}\setminus\tilde{\Lambda}_{j+1})$ for $\varepsilon^{2} N_{j}^{4C_{3}+2\alpha+5d}< 1$.
Thus, it is left to estimate
\begin{equation}
    \Pi_{0}\left(\Gamma\cap \left[ (\tilde{\Lambda}_{j}\setminus\tilde{\Lambda}_{j+1})\cap W \right]\right)
    \label{yaogujiyayayayaya}
\end{equation}
for 
\begin{equation}
    \varepsilon^{2} N_{j}^{4C_{3}+2\alpha+5d}\ge 1.
    \label{youxianbutiaojianfan}
\end{equation}

In the sequel of this section, we assume $\varepsilon^{2} N_{j}^{4C_{3}+2\alpha+5d}\ge 1$. We suppose $\alpha$ is small enough.
Fix $p_{0}\in \Pi_{0}(\Gamma\cap \tilde{\Lambda}_{j}\cap W)$. By proposition \ref{qiterative}, we have that, for $(\bm{p}(p_{0}), \lambda(p_{0}))$ and $j'<j$,
\begin{align}
    &\lVert \tilde{T}_{j', N_{j'+1}}^{-1} \rVert<2\exp (\log N_{j'+1})^{C_{2}},\\
    &|\tilde{T}_{j', N_{j'+1}}^{-1}(\xi,\xi')|<2e^{-\frac{1}{2}|\xi-\xi'|_{1}^{c}}, \textup{ for } |\xi-\xi'|_{1}>N_{j'+1}^{\frac{1}{2}}.
    \end{align}
Moreover, we also have
\begin{equation}
    |(\hat{u}_{j}-\hat{u}_{j'})(\xi)|<e^{-\frac{2}{5}(M^{j'})^{c}} e^{-|\xi|_{1}^{c}}.
\end{equation}
Thus, we have
\begin{equation}
    |(\tilde{T}_{j}-\tilde{T}_{j'})(\xi,\xi')|<e^{-\frac{1}{4}(M^{j'})^{c}}e^{-|\xi-\xi'|_{1}^{c}}.
\end{equation}
Let $j'=\lfloor \frac{j}{3}\rfloor$. Since
\[
|N_{j'+1}^{C}\exp{(\log N_{j'+1})^{C_{2}}}e^{-\frac{2}{5}(M^{j'})^{c}}|\ll 1,
\]
  we have
\begin{align}
    &\lVert \tilde{T}_{j, N_{j'+1}}^{-1} \rVert<2\exp (\log N_{j'+1})^{C_{2}},\label{Q01}\\
    &|\tilde{T}_{j, N_{j'+1}}^{-1}(\xi,\xi')|<2 e^{-\frac{1}{2}|\xi-\xi'|_{1}^{c}}, \textup{ for } |\xi-\xi'|_{1}>N_{j'+1}^{\frac{1}{2}},\label{Q02}
\end{align}
for $p_{0}\in \Pi_{0}(\Gamma\cap \tilde{\Lambda}_{j}\cap W)$.

To control $\tilde{T}_{j, N_{j+1}}^{-1}$, cover $Q=[-N_{j+1}, N_{j+1}]^{d+1}$ by interval $Q_{0}=[-N_{j'+1}, N_{j'+1}]$ and intervals $Q_{r}\subset \mathbb{Z}^{d+1}$ of size $N_{j+1}^{\frac{1}{3}}$ such that $\textup{dist } (0, Q_{r})>N_{j+1}^{\frac{1}{4}}$.

The above argument tells us for $p_{0}\in \Pi_{0}\Gamma\cap \tilde{\Lambda}_{j}\cap W$, we have (\ref{Q01}) and (\ref{Q02}). 

Now we estimate $\tilde{T}_{j, Q_{r}}^{-1}$.

Let
\begin{align}
    &\{(m,n)\in\mathbb{Z}^{d+1}:|-n^{2}\lambda^{2}+|m|^{2}+\rho|<N_{j+1}^{\alpha}\}\\
    \subset & \{(m,n)\in\mathbb{Z}^{d+1}:|-n^{2}\lambda^{2}+|m|^{2}|<2 N_{j+1}^{\alpha}\}=\Omega_{j}.
\end{align}
For $p_{0}\in \Pi_{0}\Gamma\cap \tilde{\Lambda}_{j}\cap W$, we have that $\lambda(p_{0})$ satisfies
\begin{equation}
    |\sum\limits_{k=0}^{10d} a_{k} \lambda^{k}(p_{0})|>N_{j+1}^{-C}
\end{equation}
for all $|\bm{a}|<N_{j+1}^{C}$ and $\bm{a}\ne0$. Let $B=2N_{j+1}^{\alpha}$ in Lemma \ref{fenlixingdinglifornlw}. Note that
\begin{equation}
    \max\limits_{m}(\#\{(m',n):|-n^{2}\lambda^{2}+|m'|^{2}|<2N_{j+1}^{\alpha}, m'=m\})<4N_{j+1}^{\alpha}.
\end{equation}
Let $\{\xi_{k}\}\subset\Omega_{j}$ and $|\xi_{k}-\xi_{k-1}|<2 N_{j+1}^{\alpha}$. By Lemma \ref{fenlixingdinglifornlw}, we have
\begin{equation}
    \#\{\xi_{k}\}<(N_{j+1}^{\alpha}\cdot N_{j+1}^{\alpha})^{C}=N_{j+1}^{\alpha C_{6}}.
\end{equation}
Here, $C_{6}$ is a constant depending on $d$.
Thus, there exists a partition of $\Omega_{j}$, we denote it by $\{\Omega_{\beta}\}$ such that
\begin{align}
    &\textup{diam } \Omega_{\beta}<N_{j+1}^{\alpha C_{6}+\alpha},\\
    &\textup{dist }(\Omega_{\beta}, \Omega_{\beta'})> 2N_{j+1}^{\alpha}, \textup{ for } \beta\ne\beta'.
\end{align}
Recall 
\begin{equation}
    \lambda^{2}=|m_{0}|^{2}+\rho+\frac{3}{4}\langle m_{0} \rangle^{\alpha} p_{0}^{2}\varepsilon^{2}+O(\varepsilon^{\frac{17}{8}}).
    \notag
\end{equation}
Let $p_{0}'\in B(p_{0}, \varepsilon^{-2}N_{j+1}^{-3})$. For $(m,n)\in Q\setminus\Omega_{j}$, we have
\begin{align}
    |-n^{2}\lambda^{2}(p_{0}')+|m|^{2}|>& |-n^{2}\lambda^{2}(p_{0})+|m|^{2}|-n^{2}|\lambda^{2}(p_{0})-\lambda^{2}(p_{0}')|\notag\\
    >& \frac{3}{2} N_{j+1}^{\alpha}.
\end{align}
Thus, the partition $\{\Omega_{\beta}\}$ can be applied to $B(p_{0}, \varepsilon^{-2}N_{j+1}^{-3})$.
Consider $Q_{r}\cap \Omega_{\beta}$. Let $\tilde{\Omega}_{r, \beta}$ be the $N_{j+1}^{\alpha/2}$ neighborhood of $Q_{r}\cap \Omega_{\beta}$. Now we consider $\tilde{T}_{j,\tilde{\Omega}_{r, \beta}}$.
Let $E_{r,\beta, s}(p_{0})=E_{r,\beta, s}(\bm{p}(p_{0}, \lambda(p_{0})))$ be the eigenvalue of $\tilde{T}_{j,\tilde{\Omega}_{r, \beta}}$ ($s$ refers to the index of eigenvalue).
By definition of $\tilde{\Omega}_{r, \beta}$, there exists $(m_{0}, n_{0})\in \tilde{\Omega}_{r, \beta}$ such that
\begin{equation}
    |-n_{0}^{2}\lambda^{2}+m_{0}^{2}|<2N_{j+1}^{\alpha}\textup{ and } |(m_{0}, n_{0})|>N_{j+1}^{1/4}.
\end{equation}
Thus, we have
\begin{equation}
    |n_{0}|>N_{j+1}^{1/5}.
\end{equation}
For all $(m, n)\in \tilde{\Omega}_{r, \beta}$, we have
\begin{equation}
    |n|>N_{j+1}^{1/5}-N_{j+1}^{\alpha C_{6}+2\alpha}>N_{j+1}^{1/6}.
\end{equation}
Recall that we have
\begin{equation}
    \tilde{T}_{j}=\textup{diag}\left(\frac{-(n\lambda)^{2}+\mu_{m}^{2}}{n\cdot 0+\langle m\rangle^{\alpha}}\right)+\varepsilon^{2} S_{j}.
\end{equation}
Thus, we have
\begin{equation}
    \frac{\partial \tilde{T}_{j}}{\partial \lambda}=\textup{diag } \left(\frac{-2n^{2}\lambda}{\langle  m \rangle^{\alpha}}\right)+O(\varepsilon^{2}).
\end{equation}
By the first-order eigenvalue variation, we have
\begin{equation}
    |\frac{\partial E_{r,\beta, s}}{\partial\lambda}|\gtrsim N_{j+1}^{\frac{1}{6}-\alpha}.
\end{equation}
Moreover, we have
\begin{equation}
    \frac{\partial E_{r,\beta, s}}{\partial p}=O(\varepsilon^{2}), \ \ \  \ \frac{\partial \lambda}{\partial p_{0}}\sim \varepsilon^{2},\ \ \  \ \frac{\partial \bm{p}}{\partial p_{0}} \sim 1.
\end{equation}
Thus, $\frac{\partial E_{r,\beta, s}}{\partial p_{0}}=\frac{\partial E_{r,\beta, s}}{\partial\lambda}\cdot \frac{\partial \lambda}{\partial p_{0}}+\frac{\partial E_{r,\beta, s}}{\partial\bm{p}}\cdot \frac{\partial \bm{p}}{\partial p_{0}} $ implies
\begin{equation}
    |\frac{\partial E_{r,\beta, s}}{\partial p_{0}}|\gtrsim N_{j+1}^{\frac{1}{6}-\alpha}\varepsilon^{2}.
\end{equation}
Thus, we have
\begin{align}
    &\textup{mes }\{p_{0}'\in B(p_{0}, \varepsilon^{-2} N_{j+1}^{-3}): \lVert \tilde{T}_{\tilde{\Omega_{r,\beta}}}^{-1} \rVert> N_{j+1}^{C_{7}}, \exists\  r, \beta  \}\notag \\
    \le & N_{j+1}^{-C_{7}}\varepsilon^{-2} N_{j+1}^{-\frac{1}{6}+\alpha} \cdot N_{j+1}^{\tilde{C}_{d}}\notag\\
    \le & N_{j+1}^{-\frac{1}{2} C_{7}} \varepsilon^{-2}.
\end{align}
Here, $C_{7}$ is a sufficiently large constant and $\tilde{C}_{d}$ is a constant depending on $d$. Note that by Lemma \ref{coupling lemma1} and Lemma \ref{coupling lemma2}, we have
\begin{align}
    &\Pi_{0}W \cap \{p_{0}'\in B(p_{0},\varepsilon^{-2} N_{j+1}^{-3} )\cap \Pi_{0} (\Gamma\cap \tilde{\Lambda}_{j}):\lVert \tilde{T}_{\tilde{\Omega}_{r,\beta}}^{-1} \rVert\le  N_{j+1}^{C_{7}}, \forall\  r,\beta, \}\\
    \subset & B(p_{0},\varepsilon^{-2} N_{j+1}^{-3} )\cap \Pi_{0} (\Gamma\cap \tilde{\Lambda}_{j+1}\cap W)
\end{align}
Thus,    we have
\begin{align}
    &B(p_{0},\varepsilon^{-2} N_{j+1}^{-3})\cap \Pi_{0}(\Gamma\cap(\tilde{\Lambda}_{j}\setminus\tilde{\Lambda}_{j+1}))\notag\\
    \subset & \{p_{0}'\in B(p_{0}, \varepsilon^{-2} N_{j+1}^{-3}): \lVert \tilde{T}_{\tilde{\Omega_{r,\beta}}}^{-1} \rVert> N_{j+1}^{C_{7}}, \exists\  r, \beta  \}.
\end{align}
Therefore, we obtain
\begin{align}
    \textup{mes }\Pi_{0}(\Gamma\cap(\tilde{\Lambda}_{j}\setminus\tilde{\Lambda}_{j+1}))\le N_{j+1}^{-\frac{1}{2} C_{7}} \varepsilon^{-2} \cdot \varepsilon^{2} N_{j+1}^{3}<N_{j+1}^{-2}.
\end{align}
This completes the proof.

\appendix
\section{Some facts about the generalized Diophantine conditions}
\label{gDCproperties}

\begin{proposition}
    Suppose that $\tau$ is sufficiently large. Let $I\subset \mathbb{R}^{\tilde{b}}$ be a closed interval. We have
    \begin{equation}
        \textup{mes } (I\setminus gDC_{d, \gamma,\tau})<C\gamma^{\frac{1}{n d}},
    \end{equation}
    where $C$ is a constant depending on $\tilde{b}$, $d$ and $I$.
    \label{gDCmeasure}
\end{proposition}
To prove Proposition \ref{gDCmeasure}, we need the following lemmas: 
\begin{lemma}
    Let $I\subset\mathbb{R}$ be an interval, and let  $f: I=(-1,1)\to \mathbb{R}$ be of class $C^{d}$ and $f^{(d)}\ge 1$ for all $t\in I$. Then, for all $\varepsilon>0$, we have
    \begin{equation}
        \textup{mes }\{ t\in I: |f(t)|<\varepsilon  \}<C \varepsilon^{\frac{1}{d}},
        \notag
    \end{equation}
    where the constant $C$ depends only on $d$.
    \label{ekdegaojiedaoshuceduguji}
\end{lemma}
(See \cite{eliasson2010kam} for a proof).
\begin{lemma}
 Let $I\subset \mathbb{R}^{\tilde{b}}$ be a closed interval and
    let $P(X)\in\mathbb{Z}[X_{1}, X_{2},...,X_{\tilde{b}}]$, $P(X)\ne 0$ of degree less than $d$. Then we have
    \begin{equation}
        \textup{mes } \{(x_{1},x_{2},..., x_{\tilde{b}})\in I :|P(x_{1},...,x_{\tilde{b}})|<\varepsilon\}<C b\varepsilon^{\frac{1}{\tilde{b}d}},
        \label{duoxiangshibudengshiceduguji}
    \end{equation}
    where the constant depends only on $I$ and d.
    \label{duyoyuanhanshugujicedu}
\end{lemma}

\begin{proof}
    Without loss of generality, we can assume $I=[0,1]^{\tilde{b}}$. We prove this lemma by induction on $\tilde{b}$. When $\tilde{b}=0$, this lemma is straightforward.  Suppose that this lemma holds for $\tilde{b}-1$. Let $d'$ denote the highest degree of $X_{1}$ in $P(X)$. Then, we have $d'<d$ by assumption. Furthermore, we can assume $d'>0$, otherwise, (\ref{duoxiangshibudengshiceduguji}) follows from the induction hypothesis. Let
    $P'=\partial_{X_{1}}^{d'} P$. Then, $P'\ne 0$. Furthermore, we have $P'(X)\in \mathbb{Z}[X_{2},...,X_{\tilde{b}}]$.  
    Denote
    \begin{equation}
        D=\{(x_{2},...,x_{\tilde{b}})\in [0,1]^{\tilde{b}-1}:|P'(x_{2},...,x_{\tilde{b}})|<\varepsilon^{ \frac{\tilde{b}-1}{\tilde{b}}}\}.
        \notag
    \end{equation}
    By the induction hypothesis, we have 
    $$\textup{mes } D<C(\tilde{b}-1)\varepsilon^{\frac{\tilde{b}-1}{\tilde{b}}\cdot\frac{1}{(\tilde{b}-1)d}}=C\cdot(b-1)\cdot \varepsilon^{\frac{1}{\tilde{b}d}}.$$
    Fix $(x_{2},...,x_{\tilde{b}})\in [0,1]^{\tilde{b}-1}\setminus D.$ Denote $f(t)=P(t,x_{2},...,x_{\tilde{b}})$. By the definition of the set $D$, we obtain 
    $ f^{(d')}(t)>\varepsilon^{\frac{\tilde{b}-1}{\tilde{b}}}$ or $ f^{(d')}(t)<-\varepsilon^{\frac{\tilde{b}-1}{\tilde{b}}}$
    for all $t\in [0,1]$. Furthermore, we have 
    $$
    \frac{f}{\varepsilon^{\frac{\tilde{b}-1}{\tilde{b}}}}>1 \textup{ or }  \frac{f}{\varepsilon^{\frac{\tilde{b}-1}{\tilde{b}}}}<-1
    $$
    for all $t\in [0,1]$.
    Thus, applying Lemma \ref{ekdegaojiedaoshuceduguji}, we obtain
    \begin{equation}
    \begin{split}
         &\textup{mes }\left\{t\in[0,1]: |f(t)|<\varepsilon\right\}\\
         =&\textup{mes }\left\{t\in[0,1]:\left|\frac{f}{\varepsilon^{\frac{\tilde{b}-1}{\tilde{b}}}}\right|<\varepsilon^{\frac{1}{\tilde{b}}}\right\}<C\varepsilon^{\frac{1}{\tilde{b}d'}}<C\varepsilon^{\frac{1}{\tilde{b}d}}.
    \end{split}
       \notag
    \end{equation}
    By Fubini's theorem, we have 
    \begin{equation}
        \textup{mes } \{(x_{1},x_{2},...,x_{\tilde{b}})\in I: |P(x_{1},...,x_{\tilde{b}})|<\varepsilon,\  (x_{2},...,x_{\tilde{b}})\in [0,1]^{\tilde{b}-1}\setminus D\}<C \varepsilon^{\frac{1}{\tilde{b}d}}
    \end{equation}
    In conclusion, we have
    \begin{equation}
    \begin{split}
        &\textup{mes } \{(x_{1},x_{2},..., x_{\tilde{b}})\in I :|P(x_{1},...,x_{\tilde{b}})|<\varepsilon\}\\
        <&\textup{mes } ([0,1]\times D) + \textup{mes } \{(x_{1},x_{2},...,x_{\tilde{b}})\in I: |P(x_{1},...,x_{\tilde{b}})|<\varepsilon,\  (x_{2},...,x_{\tilde{b}})\in [0,1]^{\tilde{b}-1}\setminus D\}\\
        <&C\cdot (\tilde{b}-1)\cdot \varepsilon^{\frac{1}{\tilde{b}d}}+C\varepsilon^{\frac{1}{\tilde{b}d}}
        =C \tilde{b} \varepsilon^{\frac{1}{\tilde{b}d}}.
    \end{split}
        \notag
    \end{equation}
    This completes the proof.
\end{proof}
\begin{proof}[\textbf{Proof of Proposition \ref{gDCmeasure}}]

    Applying Lemma \ref{duyoyuanhanshugujicedu}, we have
    \begin{equation}
    \begin{split}
        &\textup{mes } (I\setminus gDC_{d, \gamma,\tau})\\
        \le &\sum\limits_{\substack{\bm{a}_{P}\in\mathbb{Z}^{d+1}\\ \bm{a}_{P}\ne 0 }} \textup{mes }\{(x_{1},...,x_{\tilde{b}})\in I:|P(x_{1},...,x_{\tilde{b}})|\le \gamma (1+|\bm{a}_{P}|)^{-\tau}\}\\
        \le &\sum\limits_{n=1}^{\infty} \sum\limits_{\substack{|\bm{a}_{P}|=n \\ \bm{a}_{P}\in\mathbb{Z}^{d+1}}}\textup{mes }\{(x_{1},...,x_{\tilde{b}})\in I:|P(x_{1},...,x_{\tilde{b}})|\le \gamma (1+|\bm{a}_{P}|)^{-\tau}\}\\
        \le & \sum\limits_{n=1}^{\infty} n^{d+1} \cdot C b \gamma^{\frac{1}{\tilde{b}d}} n^{-\frac{\tau}{\tilde{b}d}}\\
        \le & C \gamma^{\frac{1}{\tilde{b}d}},
    \end{split}
        \notag
    \end{equation}
    where $C$ is a constant depending only on $\tilde{b}$, $d$ and $I$.
\end{proof}
\begin{remark}
    In Lemma \ref{duyoyuanhanshugujicedu}, we assume that $\tau$ is sufficiently large. We did not pursue the optimal $\tau$.
\end{remark}

\section{Proof of Lemma \ref{fenlixingdinglifornlw}}
\label{proof}
Here is the detailed proof of Lemma \ref{fenlixingdinglifornlw}.
\begin{proof}
     First, we assume $\sigma=0$. Define the operators
     \begin{equation}
         T_{\pm}:\mathbb{R}^{d+\tilde{b}}\to \mathbb{R}^{d+1}: (n, k)\mapsto (n, \pm k\cdot \lambda').
         \notag
     \end{equation}
    Thus, for $\xi=(n,k)$, we have
    \begin{equation}
        |n|^{2}-(\lambda'\cdot k)^{2}=T_{+}\xi \cdot T_{-}\xi.
        \notag
    \end{equation}
Denote 
\begin{equation}
    \Delta_{j} \xi=\xi_{j}-\xi_{j-1}.
    \notag
\end{equation}
Fix $K\in \mathbb{Z}_{+}$ and let $|j-j'|\le K$. Then we have
\begin{equation}
\begin{split}
    &2|T_{+}\xi_{j}\cdot T_{-}(\xi_{j'}-\xi_{j})|\\
    \le&|T_{+}\xi_{j}\cdot T_{-}\xi_{j}|+|T_{+}\xi_{j'}\cdot T_{-}\xi_{j'}|+|T_{+}(\xi_{j'}-\xi_{j})\cdot T_{-}(\xi_{j'}-\xi_{j})|\\
    \le& 2 B+ C \cdot K^{2} B^{2}<CK^{2} B^{2}. 
\end{split}
\notag
\end{equation}
Hence, we have
\begin{equation}
    |T_{+}\xi_{j}\cdot T_{-}\Delta_{j'}\xi|\le |T_{+}(\xi_{j}-\xi_{j'})\cdot T_{-}(\xi_{j'}-\xi_{j'-1})|+|T_{+}\xi_{j'}\cdot T_{-}(\xi_{j'}-\xi_{j'-1})|\le C K^{2}B^{2}.
    \notag
\end{equation}
Similarly, we have
\begin{equation}
     |T_{-}\xi_{j}\cdot T_{+}\Delta_{j'}\xi|\le C K^{2} B^{2}.
     \notag
\end{equation}
Let $I\subset [1,k]$ be an interval and $K$ be the minimum integer such that
\begin{equation}
    d_{1}=\textup{dim} [T_{+} \Delta_{j}\xi | j\in I]= \textup{dim} [T_{+} \Delta_{j}\xi | j\in I'],
    \notag
\end{equation}
for all $I'\subset I$, $|I'|\ge K$. 

Take $j\in I$, and let $I'$, $|I'|=K$ be an interval such that $j\in I'$. Denote
\begin{equation}
    E_{\pm}=[T_{\pm} \Delta_{j'}\xi | j'\in I']=[T_{\pm} \Delta_{j'}\xi | j'\in I]
    \notag
\end{equation}
and $\{e_{s}=T_{+} \Delta_{j_{s}}\}_{1\le s\le d_{1}}\subset [T_{+} \Delta_{j'}\xi | j'\in I']$ a basis for $E_{+}$. Then, we know that $\{\bar{e}_{s}=T_{-} \Delta_{j_{s}}\}_{1\le s\le d_{1}}$ is  a basis for $E_{-}$. Furthermore, we have
\begin{equation}
    T_{-}\xi_{j}\cdot e_{j}\le C K^{2} B^{2},
    \label{touying+pingmian}
\end{equation}
and 
\begin{equation}
    T_{+}\xi_{j}\cdot \bar{e}_{j}\le C K^{2} B^{2}.
    \label{touying-pingmian}
\end{equation}

Since the vectors $\{e_{s}=T_{+} \Delta_{j_{s}}\}_{1\le s\le d_{1}}$ are linearly independent in $\mathbb{R}^{d+1}$, it follows that $\{\Delta_{j_{s}}\xi=(\Delta_{j_{s}} n, \Delta_{j_{s}}k)| 1\le s\le d_{1}\}$ are linearly dependent and $\textup{dim} [\Delta_{j_{s}}n | 1\le s \le d_{1}-1]\ge d_{1}-1$.

By the assumption (\ref{guanyvlambdadejiashenlw})  on $\lambda'$ and $\lVert \Delta_{j_{s}}\xi \rVert<B$, we have 
\begin{equation}
    |\textup{det} (e_{s}\cdot e_{s'})_{1\le s,s'\le d_{1}}|>B^{-C},
    \label{++a}
\end{equation}
\begin{equation}
    |\textup{det} (e_{s}\cdot \bar{e}_{s'})_{1\le s,s'\le d_{1}}|>B^{-C},
    \label{--a}
\end{equation}
and
\begin{equation}
    |\textup{det} (\bar{e}_{s}\cdot \bar{e}_{s'})_{1\le s,s'\le d_{1}}|>B^{-C}.
    \label{+-a}
\end{equation}
(See Lemma 20.23 in \cite{bourgain2005green} for detailed proof).

From (\ref{++a}) and Cramer's rule, we have
\begin{equation}
    B^{C}\lVert \sum\limits_{s=1}^{d_{1}} c_{s} e_{s} \rVert\ge \max\limits_{s} |c_{s}|.
    \label{haohaoguji1}
\end{equation}
Hence, by (\ref{touying+pingmian}), we have
\begin{equation}
    |T_{-}\xi_{j}\cdot v|<B^{C} K^{2},
    \label{haohaoguji2}
\end{equation}
for any $v\in E_{+}$, $\lVert v \rVert\le 1$. Similarly, we have
\begin{equation}
    |T_{+}\xi_{j}\cdot w|< B^{C} K^{2},
    \label{haohaoguji2'}
\end{equation}
for any $w\in E_{-}$, $\lVert w \rVert\le 1$. Recall that $j\in I$ was an arbitrarily chosen element. Let $\zeta=\xi_{j}-\xi_{j'}$ for $j, j'\in I$.

It follows from (\ref{haohaoguji2'}) that
\begin{equation}
    |T_{+}\zeta \cdot w|<B^{C} K^{2}
\end{equation}
for any $w\in E_{-}$, $\lVert w \rVert\le 1$. Thus, we have
\begin{equation}
    |T_{+}\zeta\cdot \bar{e}_{s}|< B^{C}K^{2} \lVert \bar{e}_{s} \rVert< B^{C}K^{2}, \textup{ for } 1\le s\le d_{1}.
    \label{cmfazyongaaq}
\end{equation}

Since $T_{+}\zeta \in E_{+}$, we may write 
\begin{equation}
    T_{+}\zeta =\sum\limits_{s=1}^{d_{1}} c_{s} e_{s},
\end{equation}

From  (\ref{+-a}), (\ref{cmfazyongaaq}) and Cramer's rule, we have
\begin{equation}
    |c_{s}|<B^{C} K^{2}.
\end{equation}
Hence, we have
\begin{equation}
    \lVert n_{j}-n_{j'} \rVert\le \left\lVert T_{+}\zeta =\sum\limits_{s=1}^{d_{1}} c_{s} e_{s} \right\rVert \le B^{C} K^{2}.
\end{equation}
It follows that
\begin{equation}
    \textup{diam} \{n_{j}| j\in I\}< B^{C} K^{2}.
\end{equation}
Hence, 
\begin{equation}
    \#\{n_{j}| j\in I\}< B^{C} K^{2d}.
\end{equation}
By the assumption (\ref{duiyingndegeshuk}), we have
\begin{equation}
    |I|<B' B^{C} K^{2d}.
\end{equation}
By the definition of $K$, there exists an interval $I''\subset I$ such that $|I''|=K-1$ and
\begin{equation}
    \textup{dim } [T_{+} \Delta_{j}\xi| j\in I'']<d_{1}.
\end{equation}
Thus, we have
\begin{equation}
    |I|<B' B^{C} (|I''|+1)^{2d}.
\end{equation}
Starting from $I=[1,k]$, at most $d+1$ iterations is carried out. Thus, we have
\begin{equation}
\begin{split}
    k&< B' B^{C} (|I''|+1)^{2d}\\
    &< B' B^{C}( B' B^{C}(\cdots ((B'B^{C})^{2d})+1)^{2d}\cdots +1)^{2d}\\
    &<(B'B)^{C}.
\end{split}
\end{equation}

Now, we consider the case when $\sigma\ne 0$. Suppose that $\lambda'$ is irrational. 
Let $(\xi_{j}=(k_{j}, n_{j}))_{1\le j\le k}$ be a sequence  of distinct elements of $\mathbb{Z}^{b+d}$ such that, for some $\sigma\in\mathbb{R}$, for all $j$, we have
     \begin{equation}
         |(\lambda'\cdot k_{j}+\sigma)^{2}-|n_{j}|^{2}|<B-1
         \notag
     \end{equation}
     and 
     \begin{equation}
         |\xi_{j}-\xi_{j-1}|<B.
         \notag
     \end{equation}
Furthermore, assume that
\begin{equation}
    \max\limits_{n}(\#\{1\le j\le k| n_{j}=n\})<B'
    \notag
\end{equation}
Since $\lambda'$ is irrational, there exists $k_{0}$ such that $|\sigma-\bar{k}\cdot\lambda'|\ll 1$. Thus, we have
\begin{equation}
         |(\lambda'\cdot (k_{j}+\bar{k}))^{2}-|n_{j}|^{2}|<B.
         \notag
\end{equation}
By the case when $\sigma$, we know $k< (B'B)^{C}$. 
If $\lambda'$ is rational, there exists an irrational vector $\lambda''$ such that $\lVert \lambda'-\lambda'' \rVert\ll 1$. By a similar argument, we can prove the lemma when $\lambda'$ is rational. 
\end{proof}

\section{Coupling Lemma}
In this appendix, we state two coupling lemmas whose proof can be found in \cite[Lemma 5.3, Lemma 7]{bourgain1998quasi}.
\begin{lemma}
Assume $T$ satisfies 
\begin{equation}
    |T(\xi,\xi')|<e^{-|\xi-\xi'|_{1}^{c}} \textup{ for } \xi\ne\xi'.
\end{equation}
Let $\Omega\subset \mathbb{Z}^{d}$ be an interval and assume $\Omega=\mathop{\cup}\limits_{\zeta} \Omega_{\zeta}$ a covering of $\Omega$ with intervals $\Omega_{\zeta}$ satisfying
\begin{itemize}
    \item $|T_{\Omega_{\zeta}}^{-1}(\xi,\xi')|<B$.
    \item $|T_{\Omega_{\zeta}}^{-1}(\xi,\xi')|<K^{-C}$ for $|\xi-\xi'|_{1}>\frac{K}{100}$.
    \item For each $\xi\in\Omega$, there is $\zeta$ such that 
    \begin{equation}
        B_{K}(\xi)\cap \Omega=\{\xi'\in\Omega:|\xi'-\xi|_{1}\le K\}\subset \Omega_{\zeta}.
    \end{equation}
    \item $\textup{diam } \Omega_{\zeta}<C' K$ for each $\zeta$. 
\end{itemize}
Here, $C>C(d)$ and $B, K$ are numbers satisfying the relation
    \begin{equation}
        \log B<\frac{1}{100} K^{c} \textup{ and } K>K_{0}(c, C', r).
    \end{equation}
Then 
\begin{align}
    &|T_{\Omega}^{-1}(\xi,\xi')|<2B,\\
    &|T_{\Omega}^{-1}(\xi,\xi')|<e^{-\frac{1}{2}|x-y|_{1}^{c}} \textup{ for } |\xi-\xi'|_{1}>(100 C'K)^{\frac{1}{1-c}}.
\end{align}
    \label{coupling lemma1}
\end{lemma}

\begin{lemma}
    Fix some constants $\frac{1}{10}>\varepsilon_{1}>\varepsilon_{2}>\varepsilon_{3}>0$ and let $\Omega$ be a subset of the $M$-ball in $\mathbb{Z}^{d+1}$ ($M\to \infty$). Assume $\{\Omega_{\kappa}\}$ a collection of subsets of $\Omega$ satisfying 
    \begin{align}
        &\textup{diam } \Omega_{\kappa}<M^{\varepsilon_{1}},\\
        &\textup{dist } (\Omega_{\kappa}, \Omega_{\kappa'})>M^{\varepsilon_{2}} \textup{ for } \kappa\ne\kappa'.
    \end{align}
    Write $T=D+S$ ($D$ is a diagonal matrix) where 
    \begin{equation}
        \lVert S \rVert<\varepsilon,\ \  |S(\xi,\xi')|<\varepsilon e^{-|\xi-\xi'|_{1}^{c}},  
    \end{equation}
    where $c$ is sufficiently small  and
    \begin{align}
        |D(\xi)|>\rho\gg \varepsilon \textup{\ \  if\ \  } \xi\in\Omega\setminus \cup \Omega_{\kappa},\\
        \lVert (T|_{\tilde{\Omega}_{\kappa}})^{-1} \rVert<M^{C} \textup{ for all } \kappa,
    \end{align}
    where $\tilde{\Omega}_{\kappa}$ is an $M^{\varepsilon_{3}}$-neighborhood of $\Omega_{\kappa}$.
    Then 
    \begin{equation}
        \lVert (T|\Omega)^{-1} \rVert<\rho^{-1} M^{C+1},
    \end{equation}
    and \begin{equation}
        |(T|\Omega)^{-1}(\xi, \xi')|<e^{-\frac{1}{10}|\xi-\xi'|_{1}^{c}} \textup{\ \  if\ \  } |\xi-\xi'|_{1}>M^{2\varepsilon_{1}}.
    \end{equation}
\label{coupling lemma2}
\end{lemma}


\begin{thebibliography}{BBHM18}

\bibitem[BB13]{berti2013quasi}
Massimiliano Berti and Philippe Bolle.
\newblock Quasi-periodic solutions with {S}obolev regularity of {NLS} on {$\Bbb
  T^d$} with a multiplicative potential.
\newblock {\em J. Eur. Math. Soc. (JEMS)}, 15(1):229--286, 2013.

\bibitem[BB20]{BB20}
Massimiliano Berti and Philippe Bolle.
\newblock {\em Quasi-periodic solutions of nonlinear wave equations on the
  {$d$}-dimensional torus}.
\newblock EMS Monographs in Mathematics. EMS Publishing House, Berlin, [2020]
  \copyright2020.

\bibitem[BBHM18]{BBE18}
Pietro Baldi, Massimiliano Berti, Emanuele Haus, and Riccardo Montalto.
\newblock Time quasi-periodic gravity water waves in finite depth.
\newblock {\em Invent. Math.}, 214(2):739--911, 2018.

\bibitem[BBM14]{BBM14}
Pietro Baldi, Massimiliano Berti, and Riccardo Montalto.
\newblock K{AM} for quasi-linear and fully nonlinear forced perturbations of
  {A}iry equation.
\newblock {\em Math. Ann.}, 359(1-2):471--536, 2014.

\bibitem[BBM16]{BBM16kdv}
Pietro Baldi, Massimiliano Berti, and Riccardo Montalto.
\newblock K{AM} for autonomous quasi-linear perturbations of {K}d{V}.
\newblock {\em Ann. Inst. H. Poincar\'e{} C Anal. Non Lin\'eaire},
  33(6):1589--1638, 2016.

\bibitem[BBP13]{BBP13dnlw}
Massimiliano Berti, Luca Biasco, and Michela Procesi.
\newblock K{AM} theory for the {H}amiltonian derivative wave equation.
\newblock {\em Ann. Sci. \'Ec. Norm. Sup\'er. (4)}, 46(2):301--373, 2013.

\bibitem[Bou94]{bourgain1994construction}
Jean Bourgain.
\newblock Construction of quasi-periodic solutions for {H}amiltonian
  perturbations of linear equations and applications to nonlinear {PDE}.
\newblock {\em Internat. Math. Res. Notices}, (11), 1994.

\bibitem[Bou95]{bourgain1995construction}
J.~Bourgain.
\newblock Construction of periodic solutions of nonlinear wave equations in
  higher dimension.
\newblock {\em Geom. Funct. Anal.}, 5(4):629--639, 1995.

\bibitem[Bou98]{bourgain1998quasi}
J.~Bourgain.
\newblock Quasi-periodic solutions of {H}amiltonian perturbations of 2{D}
  linear {S}chr\"odinger equations.
\newblock {\em Ann. of Math. (2)}, 148(2):363--439, 1998.

\bibitem[Bou05]{bourgain2005green}
Jean Bourgain.
\newblock {\em Green's function estimates for lattice {S}chr\"{o}dinger
  operators and applications}.
\newblock Princeton University Press, Princeton, NJ, 2005.

\bibitem[BP11]{BP11duke}
Massimiliano Berti and Michela Procesi.
\newblock Nonlinear wave and {S}chr\"odinger equations on compact {L}ie groups
  and homogeneous spaces.
\newblock {\em Duke Math. J.}, 159(3):479--538, 2011.

\bibitem[CW93]{craig1993newton}
Walter Craig and C.~Eugene Wayne.
\newblock Newton's method and periodic solutions of nonlinear wave equations.
\newblock {\em Comm. Pure Appl. Math.}, 46(11):1409--1498, 1993.

\bibitem[CY00]{LY00}
Luigi Chierchia and Jiangong You.
\newblock K{AM} tori for 1{D} nonlinear wave equations with periodic boundary
  conditions.
\newblock {\em Comm. Math. Phys.}, 211(2):497--525, 2000.

\bibitem[EK10]{eliasson2010kam}
L.~Hakan Eliasson and Sergei~B. Kuksin.
\newblock K{AM} for the nonlinear {S}chr\"{o}dinger equation.
\newblock {\em Ann. of Math. (2)}, 172(1):371--435, 2010.

\bibitem[Kuk87]{Kuk87}
S.~B. Kuksin.
\newblock Hamiltonian perturbations of infinite-dimensional linear systems with
  imaginary spectrum.
\newblock {\em Funktsional. Anal. i Prilozhen.}, 21(3):22--37, 95, 1987.

\bibitem[Kuk93]{Kuk93}
Sergej~B. Kuksin.
\newblock {\em Nearly integrable infinite-dimensional {H}amiltonian systems},
  volume 1556 of {\em Lecture Notes in Mathematics}.
\newblock Springer-Verlag, Berlin, 1993.

\bibitem[Kuk98]{Kuk98kdv}
Sergei~B. Kuksin.
\newblock A {KAM}-theorem for equations of the {K}orteweg-de {V}ries type.
\newblock {\em Rev. Math. Math. Phys.}, 10(3):ii+64, 1998.

\bibitem[Kuk00]{Kuk00}
Sergei~B. Kuksin.
\newblock {\em Analysis of {H}amiltonian {PDE}s}, volume~19 of {\em Oxford
  Lecture Series in Mathematics and its Applications}.
\newblock Oxford University Press, Oxford, 2000.

\bibitem[LY10]{LY10spec}
Jianjun Liu and Xiaoping Yuan.
\newblock Spectrum for quantum {D}uffing oscillator and small-divisor equation
  with large-variable coefficient.
\newblock {\em Comm. Pure Appl. Math.}, 63(9):1145--1172, 2010.

\bibitem[LY11]{LY11dnls}
Jianjun Liu and Xiaoping Yuan.
\newblock A {KAM} theorem for {H}amiltonian partial differential equations with
  unbounded perturbations.
\newblock {\em Comm. Math. Phys.}, 307(3):629--673, 2011.

\bibitem[P\"96]{Pos96}
J\"urgen P\"oschel.
\newblock A {KAM}-theorem for some nonlinear partial differential equations.
\newblock {\em Ann. Scuola Norm. Sup. Pisa Cl. Sci. (4)}, 23(1):119--148, 1996.

\bibitem[Wan16]{wang2016energy}
W.-M. Wang.
\newblock Energy supercritical nonlinear {S}chr\"odinger equations:
  quasiperiodic solutions.
\newblock {\em Duke Math. J.}, 165(6):1129--1192, 2016.

\bibitem[Way90]{Way90}
C.~Eugene Wayne.
\newblock Periodic and quasi-periodic solutions of nonlinear wave equations via
  {KAM} theory.
\newblock {\em Comm. Math. Phys.}, 127(3):479--528, 1990.

\bibitem[Yua21]{yuan2021kam}
Xiaoping Yuan.
\newblock K{AM} theorem with normal frequencies of finite limit points for some
  shallow water equations.
\newblock {\em Comm. Pure Appl. Math.}, 74(6):1193--1281, 2021.

\end{thebibliography}
\end{document}